\documentclass[11pt,a4paper]{amsart}

\usepackage[utf8]{inputenc}
\usepackage{mathpazo}

\usepackage{amsmath,amssymb,amsfonts,amsthm,amscd,mathrsfs,mathtools}
\usepackage[leqno]{amsmath}
\numberwithin{equation}{section}

\usepackage{graphics}
\usepackage{rotating}
\usepackage{longtable}

\usepackage[all]{xy}

\usepackage{hyperref}
\usepackage[capitalise]{cleveref}

\usepackage{setspace}
\newtheorem{theorem}{Theorem}[section]
\newtheorem{lemma}[theorem]{Lemma}
\newtheorem{notation}[theorem]{Notation}
\newtheorem{proposition}[theorem]{Proposition}
\newtheorem{corollary}[theorem]{Corollary}

\newtheorem{definition}[theorem]{Definition}
\newtheorem{remark}[theorem]{Remark}
\newtheorem{fact}[theorem]{Fact}
\newtheorem{example}[theorem]{Example}
\newtheorem{observation}[theorem]{Observation}
\newtheorem{discussion}[theorem]{Discussion}
\newtheorem{question}[theorem]{Question}
\newtheorem{conjecture}[theorem]{Conjecture}
\newtheorem{hypothesis}[theorem]{Hypothesis}

\DeclareMathOperator{\Gdim}{Gdim}
\DeclareMathOperator{\codim}{codim}
\DeclareMathOperator{\Ass}{Ass}
\DeclareMathOperator{\im}{im}

\DeclareMathOperator{\coker}{coker}
\DeclareMathOperator{\grade}{grade}
\DeclareMathOperator{\Ann}{Ann}\DeclareMathOperator{\ann}{Ann}
\DeclareMathOperator{\Fitt}{Fitt}
\DeclareMathOperator{\unm}{unm}

\DeclareMathOperator{\Spec}{Spec}

\DeclareMathOperator{\rad}{rad}
\DeclareMathOperator{\Tr}{Tr}\DeclareMathOperator{\tr}{tr}

\DeclareMathOperator{\Ht}{ht}
\DeclareMathOperator{\pd}{pd}
\DeclareMathOperator{\Tor}{Tor}
\DeclareMathOperator{\Ext}{Ext}
\DeclareMathOperator{\Hom}{Hom}
\DeclareMathOperator{\depth}{depth}
\DeclareMathOperator{\id}{id}
\DeclareMathOperator{\Char}{char}
\DeclareMathOperator{\Supp}{Supp}

\DeclareMathOperator{\Syz}{Syz}

\newcommand{\fm}{\mathfrak{m}}
\newcommand{\fp}{\mathfrak{p}}
\newcommand{\fq}{\mathfrak{q}}

\newcommand{\fn}{\mathfrak{n}}

\newcommand{\p}{\mathfrak p}
\newcommand{\up}[1]{{{}^{#1}\!}}
\newcommand{\lo}{\longrightarrow}

\DeclareMathOperator{\height}{ht}

\begin{document}

\author[]{Mohsen Asgharzadeh}

\address{}
\email{mohsenasgharzadeh@gmail.com}

\title[Annihilator of $\Ext$ ]
{Annihilator of  $\Ext^g_R(R/I,R)$}


\begin{abstract}
	We investigate the higher divisorial ideal 
$D(I) := \Ann(\Ext^g_R(R/I,R))$
associated to an ideal $I$ of grade $g$. Our main focus is the containment problem $ D(I) \subseteq \overline{I}$. We show that this inclusion holds for broad classes of ideals, including unmixed ideals of finite projective dimension over $3$-dimensional quasi-normal rings, parameter ideals in quasi-Gorenstein rings, and powers of perfect ideals under suitable homological conditions.	If $R$ is Cohen--Macaulay local,  $I$ is height-unmixed, and
suppose $\pd_R(R/I)<\infty$, then we show
$ D_R(I)=I\subseteq \overline{I}.$ Conversely, we construct explicit examples demonstrating the necessity of these hypotheses. 
We develop structural properties of $D(I)$, relating it to unmixed parts, reflexive closures, symbolic powers, Frobenius closure, and trace ideals. Applications include the rigidity property of homological annihilators, a criteria for the triviality of reflexive modules and vector bundles on punctured spectra, as well as new connections among annihilators of Ext, conductor ideals, and local cohomology.
\end{abstract}

\maketitle
\tableofcontents

\section{Introduction}

The interplay between homological algebra and commutative algebra has long been a source of deep results and unexpected connections. A particularly fruitful interaction arises in the study of Ext modules and their annihilators, which form a bridge between module-theoretic properties and ideal-theoretic phenomena. For an ideal \(I\) in a commutative Noetherian ring \(R\), the modules \(\Ext^i_R(R/I,R)\) encode substantial information about both the ideal \(I\) and the ambient ring \(R\). Among these, the first nonvanishing Ext module—occurring when \(i = \grade(I,R)\)—is especially significant.

In his book on divisor theory in module categories \cite{wol}, Vasconcelos introduced the \emph{higher divisorial ideal} associated to an ideal \(I\) of grade \(g\):
\[
D(I) := \Ann\!\bigl(\Ext^g_R(R/I,R)\bigr).
\]
This invariant naturally interpolates between \(I\) and its radical \(\rad(I)\), and in Cohen–Macaulay settings, between the unmixed part \(I^{\unm}\) and \(\rad(I)\). Despite this seemingly constrained position, \(D(I)\) exhibits remarkable and often unexpected behavior, differing substantially from classical closure operations such as integral closure \(\overline{I}\), reflexive closure \(I^{**}\), and symbolic powers \(I^{(n)}\). Understanding the precise relationship between \(D(I)\) and these established closure operations forms the central motivation for this investigation.

The study of annihilators of Ext modules traces back to several major developments. The Auslander–Bridger transpose and its relationship with Ext revealed deep links between module theory and ideal theory. Subsequent work of Corso, Huneke, Katz, and Vasconcelos \cite{corso} systematically examined annihilators of Koszul homology (which agrees with the relevant Ext modules in the grade setting), establishing fundamental bounds and raising natural questions about containment relationships. Their work highlighted the intriguing position of \(D(I)\): while it always contains \(I\) and is contained in \(\rad(I)\), its relationship with the integral closure \(\overline{I}\) remains subtle and context-dependent.

Integral closure itself is a central notion in commutative algebra, with applications ranging from resolution of singularities to tight closure theory and multiplicity theory. Characterizing when homological invariants such as \(D(I)\) lie inside \(\overline{I}\) therefore connects two important strands of research: the homological properties of ideals and their geometric and analytic behavior as measured by integral dependence.

The guiding question of this paper, extending investigations initiated in \cite{corso}, is the following.

\begin{question}\label{1.1}
	Suppose \(R\) is Cohen–Macaulay and \(I\) is unmixed. Is it true that
	\(
	D(I) \subseteq \overline{I}\, ?
	\)
\end{question}

We approach this question from multiple perspectives, obtaining both positive results that identify broad classes of ideals satisfying the containment and negative results demonstrating its failure in general. Our main contributions are as follows.

\begin{enumerate}
	\item[(i)] \emph{Structural analysis of \(D(I)\).}  
	We establish foundational properties of higher divisorial ideals, including their behavior under flat base change, iteration properties, and monotonicity with respect to inclusion (see Corollary \ref{perd}). We also relate \(D(I)\) to primary decomposition through its containment between \(I^{\unm}\) and \(\rad(I)\) over Cohen–Macaulay rings (see Proposition \ref{dd}).
	
	\item[(ii)] \emph{Positive containment results.}  
	We prove that \(D(I) \subseteq \overline{I}\) holds for:
	\begin{itemize}
		\item[(a)] unmixed ideals of finite projective dimension over $3$-dimensional quasi-normal rings (see Proposition \ref{6.7});
		\item[(b)] parameter ideals in universally catenary quasi-Gorenstein rings (see Proposition \ref{2.8});
		\item[(c)] powers of perfect ideals satisfying suitable Betti number conditions (see Proposition \ref{pp});
		\item[(d)] ideals of analytic spread one (see Proposition \ref{6.8} and Corollary \ref{AP}).
\end{itemize} These results suggest to prove the following over Cohen--Macaulay rings:
		\begin{theorem}(See Theorem \ref{thm:finitepd}).
		Let $R$ be Cohen--Macaulay local, let $I$ be height-unmixed, and
		suppose $\pd_R(R/I)<\infty$. Then
		\(
		{\ D_R(I)=I.\ }
		\)
	\end{theorem}
	\item[(iii)] \emph{Counterexamples and limitations.}  
	We construct explicit examples showing that the containment may fail even in low-dimensional regular rings (see Corollary \ref{n}), thereby demonstrating the necessity of the unmixedness hypothesis in Question \ref{1.1}. Our example is significantly simpler than the one given in \cite[\S 1]{corso}, and the ring involved is $3$-dimensional.
Also, we give another example  which is monomial and  locally a complete intersection ideal on the punctured 
spectrum 
(see Example \ref{mn1})
versos to \cite[\S 1]{corso}, as they worked with binomials.
Then we show:

\begin{theorem}(See Theorem \ref{thm:counterexample}).
	For every field $k$ and every integer $g\ge1$, let $R  =\frac{k[[x,y,z,w_1,\ldots,w_g]]}{(xy,y^2,x^2z)}$ and define
$	I :=(w_1,\ldots,w_g,y,x^2)R$.
	Then $R $ is complete and Cohen--Macaulay of dimension $g+1$,
	$I $ is $P $-primary of height and grade $g$, and $R /I $ is
	Cohen--Macaulay of dimension one. Moreover,
	$$\quad D_{R}(I )  \nsubseteq\overline{I }.$$
	In particular, Question 1.1 has a negative answer.
\end{theorem}

	\item[(iv)] \emph{Connections with other ideal operations.}  
	We clarify relationships between \(D(I)\) and:
	\begin{itemize}
		\item[(a)] reflexive closure \(I^{**}\), particularly in dimension one;
		\item[(b)] trace ideals and conductor ideals (see Propositions \ref{cl} and \ref{5.4});
		\item[(c)] Frobenius and tight closure operations in positive characteristic (see, for example, Proposition \ref{612}).
	\end{itemize}
	
	\item[(v)] \emph{Applications.}  
	We derive consequences for:
	\begin{itemize}
		\item[(a)] criteria for the triviality of reflexive modules and vector bundles on punctured spectra (see Corollary \ref{76});
		\item[(b)] connections between annihilators of Ext and local cohomology (see Proposition \ref{1d});
		\item[(c)] characterizations of when \(D(I) = I\) (see Proposition \ref{d}).
	\end{itemize}

\item[(vi)] \emph{Connections to annihilator of Koszul.} 
It was asked in \cite[\S 1]{corso} that the rigidity of Koszul homology, i.e., is $\Ann(H_i(K(I;  R)))\subseteq\Ann(H_{i+1}(K(I;  R)))$? In Section \ref{s2.1}   we present both positive and negative sides of this question (see {Example} \ref{e215}).
The obstruction over zero-dimensional Gorenstein rings may read as follows: If the  Koszul rigidity of $R$ is valid, then   all Koszul homologies are faithful over $R/I$, see Corollary \ref{c213}.

\item[(vi)] \emph{Connections to other problems.} We present a connection from Question \ref{1.1} to other classical problems. For example we connect to the following and present series of new results:\begin{question}
	Under what conditions does
	\[
	\Ann_R(M_1 \otimes_R M_2)=\Ann_R(M_1)+\Ann_R(M_2)\quad(\ast)
	\]
	hold?
\end{question}In Proposition \ref{prop:trace} and when $R$ is Artinian Gorenstein, we show  $(\ast)$  holds if
and only if
\[
\tr(M\otimes_RN)=\tr(M)\cap\tr(N).
\]

\end{enumerate}

We conclude with a brief overview of the techniques used throughout the paper.

\begin{itemize}
	\item We employ homological methods such as the Auslander–Buchsbaum formula, duality theory, and properties of perfect and Cohen–Macaulay modules to analyze \(\Ext^g_R(R/I,R)\) and its annihilator.
	
	\item We use ideal-theoretic tools such as primary decomposition, integral closure theory, and trace ideal techniques to relate \(D(I)\) to other ideal operations.
	
	\item In positive characteristic, we exploit Frobenius actions to obtain additional results and examples.
	
	\item From a geometric viewpoint, particularly in applications to vector bundles, we use sheaf-theoretic interpretations and local cohomology to connect algebraic properties of \(D(I)\) with geometric properties of reflexive sheaves.
\end{itemize}

We follow \cite{HS} for terminology and notation concerning integral closure and refer to \cite{BH} for background on homological algebra and Cohen–Macaulay rings.

	\maketitle

\section{Preliminaries and  structural concepts}

Throughout the paper, $(R,\fm,k)$ denotes a commutative Noetherian local ring. All modules are assumed to be finitely generated unless otherwise specified.

\begin{fact}\label{k}
	Suppose $I$ is an ideal of grade $g$. Then
	\(
	I \subseteq \Ann\!\bigl(\Ext^{g}_R(R/I ,R)\bigr)\subseteq \rad(I).
	\)
\end{fact}

\begin{proof}
	This is well known. The module $\Ext^{g}_R(R/I,R)$ identifies with the last nonzero Koszul homology of $I$, and the annihilator of this Koszul homology is always contained in $\rad(I)$.
\end{proof}

\begin{remark}
	The assumption that $g=\grade(I,R)$ is essential; see \cite{syz,supp}.
\end{remark}

Let $I^{\unm}$ denote the unmixed part of $I$. Vasconcelos proved that over Cohen–Macaulay rings with a canonical module one has
\(
I^{\unm}\subseteq \Ann(\Ext^{g}_R(R/I ,\omega_R)).
\)
The following result extends this from the Gorenstein case.

\begin{fact}\label{usi}
	Suppose $R$ is Cohen–Macaulay and $I$ is an ideal of height $g$. Then
	\(
	\Ext^{g}_R(R/I ,R) \cong \Ext^{g}_R(R/I^{\unm} ,R).
	\)
	In particular,
	\(
	I^{\unm}\subseteq \Ann\!\bigl(\Ext^{g}_R(R/I ,R)\bigr)\subseteq \rad(I).
	\)
\end{fact}

\begin{proof}
	The second inclusion follows from Fact \ref{k}. Any prime ideal associated to $I^{\unm}$ has height equal to $\Ht(I)$, while every associated prime of $I^{\unm}/I$ has height strictly greater than $\Ht(I)$. Hence
	\[
	\grade\!\bigl(\Ann(I^{\unm}/I),R\bigr)=\Ht\!\bigl(\Ann(I^{\unm}/I)\bigr)>\Ht(I)=g.
	\]
	It follows that
	\(
	\Ext^{g-1}_R(I^{\unm}/I,R)=\Ext^{g}_R(I^{\unm}/I,R)=0.
	\)
	Applying $\Hom_R(-,R)$ to the exact sequence
	\[
	0 \longrightarrow I^{\unm}/I \longrightarrow R/I \longrightarrow R/I^{\unm} \longrightarrow 0
	\]
	yields
	\[
	0=\Ext^{g-1}_R(I^{\unm}/I,R) \longrightarrow \Ext^{g}_R(R/I^{\unm} ,R) \longrightarrow \Ext^{g}_R(R/I ,R) \longrightarrow \Ext^{g}_R(I^{\unm}/I,R)=0,
	\]
	which proves the claim.
\end{proof}

The Cohen–Macaulay hypothesis is necessary.

\begin{example}
	Let $R:=K[[x,x_1,\ldots,x_n]]/(x\fm)$ and set $I=0$. Then $I^{\unm}=xR$ has height zero. Since
	\(
	\Ext^{0}_R(R/I ,R)=\Hom_R(R,R)=R,
	\)
	we have
	\(
	\Ann(\Ext^{0}_R(R/I ,R))=0,
	\)
	so $I^{\unm}=xR\nsubseteq 0$.
\end{example}

By $(G_1)$ we mean the Gorenstein condition in codimension one, and by $(S_2)$ we mean Serre's condition. A \emph{quasi-normal} ring is one satisfying both $(G_1)$ and $(S_2)$. We write $(-)^\ast:=\Hom_R(-,R)$ and $\overline{I}$ for the integral closure of $I$.

\begin{fact}[{\cite[4.2]{Fos}}]\label{fos}
	Suppose $R$ is quasi-normal and let $M$ be a lattice. Then
	\(
	\Hom_R(M,R)^\ast = \bigcap_{P \in \Spec^1(R)} M_P .
	\)
\end{fact}

\begin{conjecture}\label{26}
	Let \( I \) be an ideal of height and grade one. Then \( I^{**} \subseteq \overline{I} \).
\end{conjecture}

We write $\pd_R(-)$ and $\id_R(-)$ for the projective and injective dimensions, respectively.

\begin{remark}
	Conjecture~\ref{26} fails if one assumes only that $\Ht(I)=1$. For an example, let 
	\(
	R:=k[[X,Y]]/(X^2,XY).
	\)
	
	\begin{itemize}
		\item Let $I=\fm$, which has height one. Then $\overline{I}=\fm$. Suppose, for contradiction, that 
		\(
		I^{**} \subseteq \overline{I}=\fm.
		\)
		Applying $(-)^{*}$ to the short exact sequence
		\(
		0\rightarrow \fm \xrightarrow{i} R \rightarrow R/\fm \rightarrow 0
		\)
		yields
		\[
		\begin{CD}
		0@>>> \Hom_R(R/\fm , R) @>>> R^\ast @>>> \fm^\ast @>>> \Ext_{R}^1(R/\fm,R) @>>> 0\\
		@. @AAA \cong @AAA = @AAA = @AAA\\
		0@>>> xR @>>> R @>>> \fm^\ast @>>> \bigoplus_{i=1}^{t} R/\fm @>>> 0,
		\end{CD}
		\]
		where $t$ is the type of $R$. From the induced maps we obtain
		\[
		0 \longrightarrow \bigoplus_{i=1}^{t} xR \longrightarrow \fm^{**} \longrightarrow (R/xR)^\ast \longrightarrow \bigoplus_{i=1}^{t} \Ext_{R}^1(R/\fm,R).
		\]
		Combining this with the previous diagram shows
		\(
		\bigoplus_{i=1}^{t} \Hom_R(R/\fm , R) \subseteq \fm^{**} \subseteq \overline{I}=\fm.
		\)
		Hence $t=1$. By a result of Roberts (see \cite[9.6.3]{BH}), this implies that $R$ is Cohen–Macaulay, a contradiction.
		
		\item Now set $I:=yR$, which also has height one. Since $x$ satisfies $x^2=0$, we have $x\in \overline{I}$, and thus $\overline{I}=\fm$. Because
		\(
		I=R/(0:_Ry)=R/xR,
		\)
		we obtain
		\(
		I^\ast=(0:_Rx)=\fm,
		\)
		and therefore $I^{**}=\fm^\ast$.
		From the exact sequence
		\(
		0\to \fm \to R\to k\to 0
		\)
		we deduce
		\[
		0\to \Hom_R(k,R)=xR \to R\to \fm^\ast\to \Ext^1_R(k,R)\to 0.
		\]
		Since $\mu^i(\fm,R)\neq 0$ for all $i\in[\depth(R),\id(R)]$, we have $\Ext^1_R(k,R)\neq 0$. Consequently, there is a short exact sequence
		\[
		0\to I=R/xR \longrightarrow I^{**} \longrightarrow k^{\oplus s} \to 0
		\]
		for some $s>0$. 
	If $I^{**}\subseteq \overline{I}=\fm$, then
		\[
		k^{\oplus s}\cong I^{**}/I \subseteq \overline{I}/I=\fm/I=(x,y)/y \cong xR \cong k,
		\]
		forcing $s=1$. Thus the type of $R$ is one, which again implies that $R$ is Cohen–Macaulay, a contradiction.
	\end{itemize}
\end{remark}
 In section 4 we give a negative answer to 	Conjecture~\ref{26}.

\begin{proposition}\label{2.8}
	Let $R$ be universally catenary and quasi-Gorenstein, and let $I$ be a parameter ideal of $g:=\codim(I)$. Then
	\(
	\Ann\!\bigl(\Ext^g_R(R/I, R)\bigr) \subseteq \overline{I}.
	\)
\end{proposition}

\begin{proof}
	Note that $g=\dim R-\dim(R/I)$. First assume $g=\dim R$, so $I$ is $\fm$-primary and $I^{\unm}=I$. Since $R$ is quasi-Gorenstein,
	\(
	\Ext^g_R(R/I, R)=\Ext^g_R(R/I, K_R)=K_{R/I}
	\)
	(see \cite{Ao}). Hence $\Ann(K_{R/I})=U_{R/I}=I$, and the claim follows.
	Now assume $g<\dim R$. By \cite[Ex.\ 10.11]{Hut}, the universally catenary hypothesis implies $I^{\unm}\subseteq\overline{I}$. Therefore,
	\(
	\Ann(\Ext^g_R(R/I, R)) \subseteq I^{\unm} \subseteq \overline{I}.
	\)
\end{proof}

By $I^\star$ we denote the tight closure of $I$.

\begin{corollary}
	Let $R$ be of prime characteristic, essentially of finite type over a field, and quasi-Gorenstein. If $I$ is a parameter ideal with $g:=\codim(I)<\dim R$, then
	\(
	\Ann(\Ext^g_R(R/I, R)) \subseteq I^\star.
	\)
\end{corollary}

\begin{remark}
	Let $R$ be $F$-regular of dimension $d\ge 2$. Then there exists an ideal $I$ of grade $g$ such that
	\[
	\Ann(\Ext^g_R(R/I, R)) \nsubseteq I^\star.
	\]
\end{remark}

\begin{proof}
	Since $d>1$, there exists an ideal $I$ that is not unmixed. If
	\(
	\Ann(\Ext^g_R(R/I, R)) \subseteq I^\star,
	\)
	then by Fact \ref{usi},
	\(
	I^{\unm}\subseteq \Ann(\Ext^g_R(R/I, R)) \subseteq I^\star = I,
	\)
	because $F$-regular rings have all ideals tightly closed. Thus $I$ would be unmixed, a contradiction.
\end{proof}

\begin{proposition}
	Let $I$ be a height-one ideal and suppose $R$ is normal. Then
	\(
	I^{**} \supseteq \overline{I}.
	\)
\end{proposition}

\begin{proof}
	We have
	\(
	I^{**} = \bigcap_{\fp\in \Spec^1(R)} IR_{\fp} \cap R.
	\)
	For each $\fp\in\Spec^1(R)$, the localization $R_{\fp}$ is a DVR by $(R_1)$, hence integrally closed. Let $\Sigma$ denote the set of DVRs birationally dominating $R$. Then
	\[
	I^{**} = \bigcap_{\fp\in \Spec^1(R)} IR_{\fp} \cap R
	\supseteq \bigcap_{V\in\Sigma} IV \cap R
	= \overline{I},
	\]
	which completes the proof.
\end{proof}

	\maketitle

	\maketitle

\section{Higher divisorial ideals}

We begin by introducing the main object of this paper.

\begin{definition}[Vasconcelos \cite{wol}]
	Let $I$ be an ideal of grade $g$. The \emph{higher divisorial ideal} of $I$ is defined by
	\(
	D(I):=D_R(I):=\Ann_R\!\big(\Ext^{g}_R(R/I,R)\big).
	\)
\end{definition}

\begin{observation}
	Let $f:R\to S$ be a flat homomorphism of local rings and let $I$ be an ideal of $R$. Then
	\(
	D_R(I)S = D_S(IS).
	\)
\end{observation}

\begin{proof}
	Suppose $\grade(I)=g$. Then
	\(
	\Ext^i_R(R/I,R)\otimes_R S \cong \Ext^i_S(S/IS,S)
	\)
	for all $i$. Hence $\Ext^i_S(S/IS,S)=0$ for $i<g$, while
	\(
	\Ext^{g}_S(S/IS,S)\cong \Ext^{g}_R(R/I,R)\otimes_R S \neq 0
	\)
	since $S$ is faithfully flat. Thus $\grade(IS)=g$. Moreover,
	\[
	\Ann_S\!\big(\Ext^{g}_S(S/IS,S)\big)
	= \Ann_S\!\big(\Ext^{g}_R(R/I,R)\otimes_R S\big)
	= \Ann_R\!\big(\Ext^{g}_R(R/I,R)\big)S,
	\]
	because annihilators commute with flat base change.
\end{proof}

\begin{fact}\label{ddd}
	\begin{enumerate}
		\item[(i)] $I \subseteq D(I) \subseteq \rad(I)$.
		\item[(ii)] If $R/I$ is Noetherian, then there exists $n>0$ such that
		$D^n(I)=D^m(I)$ for all $m\ge n$.
		\item[(iii)] If $\grade(I)=1$, then $D(D(I))=D(I)$.
	\end{enumerate}
\end{fact}

\begin{proof}
	(i) is given in Fact~\ref{k}.  
	
	(iii) If $\grade(I)=1$, then $D(I)=I^{**}$. Since reflexive hulls are idempotent, $I^{****}=I^{**}$, and the result follows.
	
	(ii) Define $D^{\ell}(I):=D(D^{\ell-1}(I))$ for $\ell>1$. By (i) we have an ascending chain
	\[
	I \subseteq D(I)\subseteq D^2(I)\subseteq \cdots.
	\]
	Since $R/I$ is Noetherian, this chain stabilizes.
\end{proof}

\begin{proposition}\label{d0d}
	One has $D(D(I))=D(I)$.
\end{proposition}

\begin{proof}
	First assume $\depth R=0$. By definition,
	\[
	D(I)=\Ann(\Hom_R(R/I,R))=\Ann(0:I)=(0:(0:I)).
	\]
	Thus
	\[
	D^2(I)=(0:(0:D(I)))=(0:(0:(0:(0:I)))).
	\]
	Recall that $\Ann(J)=\Ann(\Ann(\Ann(J)))$ for any ideal $J$ (see \cite[2.7]{lindo}). Hence $D^2(I)=D(I)$.
Now consider the general case. Let $\underline{x}=x_1,\dots,x_g$ be a maximal regular sequence contained in $I$. Then
	\(
	\underline{x}R \subseteq I \subseteq D(I)\subseteq D^2(I).
	\)
	It suffices to show
	\(
	D(I)/\underline{x}R = D^2(I)/\underline{x}R.
	\)
	But
	\[
	D(I)=(\underline{x}:(\underline{x}:I)), \qquad
	D^2(I)=(\underline{x}:(\underline{x}:(\underline{x}:(\underline{x}:I)))).
	\]
	Thus the claim reduces to proving
	\[
	(0:(0:I)) = (0:(0:(0:(0:I))))
	\]
	inside $R/\underline{x}R$, which follows from the depth-zero case.
\end{proof}

\begin{corollary}\label{perd}
	If $I_1\subseteq I_2$ are ideals of the same grade, then $D(I_1)\subseteq D(I_2)$. In particular,
	\[
	\bigcap_{m=1}^{\infty}D(I^m)\subseteq \cdots \subseteq D(I^{n+1})\subseteq D(I^n)\subseteq \cdots \subseteq D(I).
	\]
\end{corollary}

\begin{proof}
	We reduce to the case $\grade(I_i)=0$. Then $(0:I_1)\supseteq (0:I_2)$, so
	\(
	D(I_1)=\Ann(0:I_1)\subseteq \Ann(0:I_2)=D(I_2).
	\)
\end{proof}

\begin{example}\label{36}
	\begin{enumerate}
		\item[(i)] There exist $I_1\subseteq I_2$ such that $D(I_1)\nsubseteq D(I_2)$.
		
		\item[(ii)] If $x$ is $R$-regular, then $\bigcap_{m\ge1} D((xR)^m)=0$.
		
		\item[(iii)] If $R$ is Cohen--Macaulay and $I$ is $\fm$-primary, then $\bigcap_{m\ge1} D(I^m)=0$.
		
		\item[(iv)] If $R$ is regular of prime characteristic, then $\bigcap_{m\ge1} D(I^m)=0$.
		
		\item[(v)] Let $R=\dfrac{k[[x,x_1,\dots,x_n]]}{x\fm}$. Then $D(-)$ is constant on nonzero ideals.
		
		\item[(vi)] In the same ring as in (v), $\bigcap_{m\ge1} D(I^m)\ne0$ for every nonzero ideal $I$.
		
		\item[(vii)] If $R$ is a one-dimensional domain and $I\ne0$, then $\bigcap_{m\ge1} D(I^m)=0$.
	\end{enumerate}
\end{example}

\begin{proof}
	(i) Let $R=k[[x,y]]$, $I_1=(x^2,xy)$ and $I_2=(x^2,y)$. Then $I_1\subseteq I_2$ and
	$I_1=(x)\cap(x^2,y)$, so $D(I_1)=I_1^{**}=I_1^{\unm}=(x)$. Since $I_2$ is $\fm$-primary and integrally closed,
	$I_2=D(I_2)$. Thus $D(I_1)\nsubseteq D(I_2)$.
	
	(ii) Since $\Ext^1(R/x^mR,R)\cong R/x^mR$, we get $D((xR)^m)=(x^m)$. Krull's intersection theorem gives the result.
	
	(iii) Because $D(I^m)\subseteq \overline{I^m}$, it suffices to show $\bigcap_m \overline{I^m}=0$. This follows either from Briançon–Skoda plus Krull intersection, or from the valuative criterion: if $x$ lies in all $\overline{I^m}$, then $v(x)>m$ for all $m$, so $x=0$.
	
	(iv) The intersection annihilates $\Ext^g(R/I^m,R)$ for all $m$, hence annihilates $H^g_I(R)$. Since $\Ann(H^g_I(R))=0$, the claim follows.
	
	(v) If $0\ne a\in\fm$, then $\grade(aR)=0$ and
	$D(aR)=\Ann(0:a)=\Ann(x)=\fm$. By Corollary~\ref{perd}, $D(I)=\fm$ for any nonzero $I\subseteq\fm$.
	
	(vi) If $I\subseteq\fm$, then $D(I^m)=\fm$ for all $m$, so the intersection equals $\fm$.
	
	(vii) This follows from (iii).
\end{proof}
\begin{proposition}\label{dd} 
	Suppose $R$ is Cohen--Macaulay. The following assertions hold:
	\begin{enumerate}
		\item[(i)] 
		\(
		I \subseteq I^{\unm} \subseteq D(I) \subseteq \rad(I).
		\)
		\item[(ii)]
		Suppose $I$ and $J$ are primary ideals of the same height with different associated primes. Then
		\(
		D(I \cap J) \subseteq D(I) \cap D(J).
		\)
	\item[(iii)] Suppose in addition  to (ii), $I$ and $J$ be	complete-intersection. Then \(
	D(I \cap J ) = D(I) \cap D(J ).
	\)\end{enumerate}
\end{proposition}

\begin{proof}
	(i) This is well known; see Fact~\ref{usi}.
	
	(ii) Consider the exact sequence
	\(
	0 \to R/(I \cap J) \to R/I \oplus R/J \to R/(I + J) \to 0.
	\)
	Applying $\Ext^g_R(-,R)$ gives
	\[
	\Ext^g_R(R/(I + J), R) \to \Ext^g_R(R/I, R) \oplus \Ext^g_R(R/J, R) \to \Ext^g_R(R/(I \cap J), R).
	\]
	
	Let $P \in \min(I+J)$. If $P \in \min(I)$, then $P \notin \min(J)$, so $P \supseteq J$ and hence $\Ht(P) > \Ht(J)$. The same conclusion holds if $P \in \min(J)$. If $P \notin \min(I) \cup \min(J)$, then again $P \supseteq I$ and $P \supseteq J$, so $\Ht(P) > \Ht(I)=\Ht(J)$. Thus every minimal prime of $I+J$ has height strictly larger than $g=\Ht(I)=\Ht(J)$. Since $R$ is Cohen--Macaulay, this implies
	\(
	\Ext^g_R(R/(I+J),R)=0.
	\) Therefore,
	\[
	\Ext^g_R(R/I,R)\oplus \Ext^g_R(R/J,R)
	\subseteq \Ext^g_R(R/(I\cap J),R),
	\]
	and hence
	\[
	\begin{aligned}
	D(I)\cap D(J)
	&= \Ann\!\big(\Ext^g_R(R/I,R)\big)\cap \Ann\!\big(\Ext^g_R(R/J,R)\big) \\
	&= \Ann\!\big(\Ext^g_R(R/I,R)\oplus \Ext^g_R(R/J,R)\big) \\
	&\stackrel{(+)}\supseteq \Ann\!\big(\Ext^g_R(R/(I\cap J),R)\big) = D(I\cap J).
	\end{aligned}
	\]
	
	(iii) It follows that $(+)$ from (ii) is indeed an equality, and so the claim follows.
\end{proof}

\begin{example}\label{ex:planes}
	Let \(R = k[[x,y,z,w]]\),  let
	\(
	I = (x^2,y)\) and \(J = (z^2,w).
	\)
	Then
\(
D(I \cap J ) = D(I) \cap D(J ).
\)
\end{example}
	
	\begin{proof}
 From the exact sequence \(0 \to R/I \cap J \to R/I \oplus R/J \to R/L \to 0\) with $L$ is $\fm$-primary, we get
		 the  claim.
	\end{proof}

\begin{proposition}\label{d}
	The following statements hold:
	\begin{enumerate}
		\item[(i)] If $R$ is Cohen--Macaulay, then $D(\fm^n)\subseteq\overline{\fm^n}$ for all $n$.
		\item[(ii)] If $R$ is regular, then $D(\fm^n)=\fm^n$ for all $n$.
		\item[(iii)] Suppose $R$ is zero-dimensional. Then $D_R(-)=\id(-)$ if and only if $R$ is Gorenstein.
	\end{enumerate}
\end{proposition}

\begin{proof}
	(i) and (ii): By Proposition~\ref{dd}(i), we have $D(\fm^n)\supseteq \fm^n$. Since $\fm^n$ is $\fm$-primary, it follows from \cite[2.15]{corso} that
	\(
	D(\fm^n)\subseteq\overline{\fm^n}.
	\)
	If $R$ is regular, then powers of the maximal ideal are integrally closed (see \cite[Ex.~1.18]{HS}). Hence
	\(
	\fm^n \subseteq D(\fm^n) \subseteq \overline{\fm^n}=\fm^n,
	\)
	so equality holds.
	
	(iii) Recall that $\Hom_R(R/I,R)=(0:_R I)$. Hence
	\(
	D(I)=\Ann\big((0:_R I)\big).
	\)
	Thus $D(I)=I$ for all ideals $I$ if and only if $\Ann(\Ann(I))=I$ for all $I$, which is equivalent to $R$ being Gorenstein (see \cite[Ex.~3.2.15]{BH}).
\end{proof}

\begin{example}\label{ex1}
	The regularity assumption in Proposition~\ref{d}(ii) is essential.
Let
$	R = k[x^5,x^6,x^{14}] \subset k[x], 
$ and $I=\fm^2.
$ Then $\fm^2$ is not integrally closed and
	\(
	\fm^2 \subsetneq D(\fm^2) \subsetneq \overline{\fm^2}.
	\)
\end{example}

\begin{proof}
	From the exact sequence
$
	0 \to I \to R \to R/I \to 0
$
	we obtain
	\[
	0 \longrightarrow \Hom_R(R,R)
	\longrightarrow \Hom_R(I,R)
	\longrightarrow \Ext^1_R(R/I,R)
	\longrightarrow 0.
	\]
	Hence
$
	\Ext^1_R(R/I,R)
	\cong \Hom_R(I,R)/R
	\cong I^{-1}/R.
$	Since $R$ is integrally closed in $k[x]$, we have
	\[
	I^{-1} \subseteq k[x],
	\qquad
	\Ext^1_R(R/I,R)\subseteq k[x]/R.
	\]
	Thus
	\(
	\Ann\!\big(\Ext^1_R(R/I,R)\big) \supseteq \Ann(k[x]/R).
	\) Because $x^5\cdot x^4 \notin R$, we have $x^5 \notin \Ann(k[x]/R)$. On the other hand,
	\( 
	\{x^6,x^{14}\} \subset \Ann(k[x]/R),
	\)
	so
	\(
	(\fm^2,x^6,x^{14}) \subset \Ann\!\big(\Ext^1_R(R/I,R)\big).
	\)
	Therefore,
$
	\Ann\!\big(\Ext^1_R(R/I,R)\big) \subsetneq \overline{\fm^2}.
$ Finally, $\fm^2$ is not reflexive (see \cite[7.5]{dao}), hence it is not integrally closed, since over a one-dimensional ring integrally closed ideals are reflexive (see \cite[2.14]{corso}).
\end{proof}
In Proposition~\ref{d}(i) and (ii), the powers of the maximal ideal behave in a particularly nice way. The next result shows that this behavior fails for powers of prime ideals in general.

\begin{corollary}\label{n}
	There exists a Cohen–Macaulay ring $R$ and a power $I$ of a prime ideal such that
	\(
	D(I)\nsubseteq \overline{I}.
	\)
\end{corollary}

\begin{proof}
	Let $k$ be a field of characteristic $2$ and set
	\(
	R:=k[[X,Y,Z]].
	\)
	Let $P$ be the defining ideal of the monomial curve $k[[t^3,t^4,t^5]] \subseteq k[[t]]$. Then $P\in\Spec(R)$ and
	\[
	g:=\grade(P,R)=\Ht(P)=2.
	\]
	The same equalities hold for $I:=P^2$.
	From \cite[3.7]{h} we know:
	\begin{enumerate}
		\item[(i)] $P^2 \ne P^{(2)}$,
		\item[(ii)] $P^2 = \overline{P^2}$.
	\end{enumerate}
	
	Suppose, toward a contradiction, that
	\(
	D(P^2)=\Ann\!\big(\Ext^g_R(R/P^2,R)\big)\subseteq \overline{P^2}.
	\)
	Since $P^n \subseteq P^{(n)}$ for all $n$, we have $P^2 \subseteq P^{(2)}$. Now we compare these ideals using reflexive hulls.
	Because $R$ is regular, it is quasi-normal. Hence by Fact~\ref{fos},
	\(
	(P^2)^{**}=\bigcap_{\fp\in \Spec^1(R)} (P^2)R_\fp = P^{(2)}.
	\)
	Moreover, since $R$ is Cohen–Macaulay and $\Ht(P^2)=g$, Fact~\ref{usi} gives
	\[
	(P^2)^{**}\subseteq \Ann\!\big(\Ext^g_R(R/P^2,R)\big)=D(P^2).
	\]
	Combining these containments with the assumption and (ii), we obtain
	\[
	P^{(2)}=(P^2)^{**}\subseteq D(P^2)\subseteq \overline{P^2}=P^2\subseteq P^{(2)}.
	\]
	Hence $P^{(2)}=P^2$, contradicting (i). Therefore
	\(
	D(P^2)\nsubseteq \overline{P^2}.
	\)
\end{proof}

\begin{remark}
	The above ideal $P^2$ is not unmixed. In fact, if $I$ is an unmixed ideal of finite projective dimension in a ring $R$ with $\dim(R/I)\le 1$, then one has
	\(
	D(I)\subseteq \overline{I}.
	\)
\end{remark}
	The following example illustrates why the unmixedness hypothesis in our main results cannot simply be replaced by the weaker condition of being locally a complete intersection on the punctured spectrum.
\begin{example}\label{mn1} 
Let
$
	R = k[[x,y,z]],
$
	be a regular local ring of dimension \(3\), and define
$
	I = (x) \cap \mathfrak m^2.
$
Then \(I_{\mathfrak p}\) is a complete intersection for every
	\(
	\mathfrak p \in V(I) \setminus \{\mathfrak m\},
	\)
	but
	\( 
 {D(I) \nsubseteq \overline{I}}.
	\)
\end{example}

\begin{proof}
We have
$
	\sqrt{I} = (x).
$
	If \(\mathfrak p \neq \mathfrak m\) and \(I \subseteq \mathfrak p\), then \(x \in \mathfrak p\). Moreover,
$
	I_{\mathfrak p} = x\mathfrak m_{\mathfrak p}.
$
	Because \(\mathfrak p \neq \mathfrak m\), at least one of \(x, y, z\) is a unit in \(R_{\mathfrak p}\). Since \(x \in \mathfrak p\), either \(y\) or \(z\) is a unit. Hence
$
	\mathfrak m_{\mathfrak p} = R_{\mathfrak p},
$
	and therefore
$
	I_{\mathfrak p} = (x)R_{\mathfrak p}.
$
	Thus \(I_{\mathfrak p}\) is generated by a regular sequence of length \(1\) for every
$
	\mathfrak p \in V(I) \setminus \{\mathfrak m\}.
$
	So \(I\) is locally a complete intersection on the punctured spectrum.
Here, is a computation of \(D(I)\). 
Recall that
$	g = \operatorname{grade}(I) = 1.
$
	The unmixed part of \(I\) is
$
	I^{\rm unm} = (x).
$
	Since \(R\) is regular (hence Cohen--Macaulay), Fact \ref{usi}  gives
	\[
	\operatorname{Ext}^1_R(R/I, R) \cong \operatorname{Ext}^1_R(R/I^{\rm unm}, R)
	= \operatorname{Ext}^1_R(R/(x), R).
	\]
So
$
	\operatorname{Ext}^1_R(R/(x), R) \cong R/(x).
$
	Consequently,
	$
	D(I) = \operatorname{Ann}_R \operatorname{Ext}^1_R(R/I, R)
	= \operatorname{Ann}_R(R/(x))
	= (x).
$
Next we show \(x \notin \overline{I}\).
Indeed, using the \(\mathfrak m\)-adic order \(v = \operatorname{ord}_{\mathfrak m}\), we have
	\(
	v(x) = 1,
	\)
	whereas every element of \(I = x\mathfrak m\) has order at least \(2\):
	\(
	v(I) = 2.
	\)
	If \(x \in \overline{I}\), then every Rees valuation of \(I\) would satisfy
	\(
	v(x) \ge v(I).
	\) Hence
	\(
	x \notin \overline{I}.
	\)
	Therefore
$D(I) = (x) \nsubseteq \overline{I}$.
\end{proof}
In Proposition~\ref{d}(i), the Cohen--Macaulay assumption is essential:

\begin{remark}
	Let $R:=k[[X,Y]]/(X^2,XY)$ and let $n>0$. Then 
	\(
	D(\fm^n) \subseteq \overline{\fm^n}
	\quad\text{if and only if}\quad n=1.
	\)
\end{remark}

\begin{proof}
	First note that $\fm=(x,y)R$ and $x^2=xy=0$ in $R$. Hence $x\in (0:_R\fm)$, so
	\[
	D(\fm)=\Ann\!\big(\Hom_R(R/\fm,R)\big)=\Ann(0:_R\fm)=\Ann(x)=\fm.
	\]
	Since $\fm=\overline{\fm}$, the statement holds for $n=1$.
	Now assume $n>1$. Because $x^2=xy=0$, every product involving $x$ and at least one additional factor vanishes. Thus, for $n\ge 2$ we have
	\(
	\fm^n=(x,y)^nR=y^nR.
	\)
	Moreover, $\depth(R)=0$, so $\grade(\fm^n,R)=0$. Hence
	\(
	\Hom_R(R/\fm^n,R)=(0:_R \fm^n)=(0:_R y^n).
	\)
	Since $xy^n=0$ and $y^n\neq 0$, we obtain $(0:_R y^n)=xR$. Therefore
	\[
	D(\fm^n)=\Ann(xR)=(0:_R x)=\fm.
	\]
	
	On the other hand, $\fm^n=y^nR$ is integrally closed (being principal in the reduced ring $R/xR\cong k[[Y]]$), so
	\(
	\overline{\fm^n}=\fm^n=y^nR.
	\)
	Consequently,
	\(
	D(\fm^n)=\fm \nsubseteq y^nR=\fm^n=\overline{\fm^n}
	\)
	for all $n>1$, which completes the proof.
\end{proof}

	\maketitle
\section{Failure of 	\(
	D(I) \subseteq \overline{I} 
	\) for unmixed ideals}
We first compute in a one-dimensional ring.

\begin{lemma}\label{lem:base}
	Let $k$ be any field, 
$T=k[[x,y,z]]$, $ Q=(xy,y^2,x^2z),$ $ A=T/Q$,
and $ L=(y,x^2)A.$
	Then $A$ is a one-dimensional Cohen--Macaulay local ring, $L$ is
	$(x,y)A$-primary, and

	\begin{enumerate}
	\item[(i)] $\Ann_A(L)=(y,xz)A$,
	\item[(ii)] $	\Ann_A\Ann_A(L)=(x,y)A$.
\end{enumerate}

	Furthermore $A/L\cong k[[x,z]]/(x^2)$ is Cohen--Macaulay of
	dimension one.
\end{lemma}

\begin{proof}
	In $T$ one has the primary decomposition
\(Q=(x,y)^2\cap(y,z).
\)
	Indeed, intersecting these monomial ideals gives the minimal
	generators $xy,y^2,x^2z$. The first component is $(x,y)$-primary
	and the second is prime. Both primes have height two. Thus
	$\dim A=1$ and every associated prime of $A$ belongs to
	$\{(x,y)A,(y,z)A\}$. The element $x+z$ belongs to neither primes,
	so it is $A$-regular. Consequently $\depth A=\dim A=1$.
	
	The quotient by $L$ is $k[[x,z]]/(x^2)$. Writing its elements uniquely
	as $a(z)+xb(z)$, with $a(z),b(z)\in k[[z]]$, shows that every element
	outside $(x)$ is a nonzerodivisor. Every element of $(x)$ has square
	zero. Therefore $(x^2)$ is $(x)$-primary in $k[[x,z]]$, and $L$ is
	$(x,y)A$-primary. Multiplication by $z$ is injective on this quotient,
	proving its Cohen--Macaulayness.
	
	For the annihilators, compute the following monomial colons in $T$:
	\[
	(Q:y)=(x,y),\qquad (Q:x^2)=(y,z),\qquad (Q:xz)=(x,y).
	\]
	Hence
	\begin{enumerate}
	\item[(i)'] $Q:(y,x^2)=(Q:y)\cap(Q:x^2)=(y,xz)$,
	\item[(ii)'] $	Q:(y,xz)=(Q:y)\cap(Q:xz)=(x,y)$.
\end{enumerate}	
	Passing to $A=T/Q$ proves (i) and (ii).
\end{proof}

We recall the following well-known fact, but our prove differ with the classical ones.
\begin{lemma} \label{lem:rees}
	Let $u_1,\ldots,u_g$ be an $R$-regular sequence, let
	$B=R/(u_1,\ldots,u_g)$, and let $M$ be a finite $B$-module.
	Then
$\Ext_R^g(M,R)\cong\Hom_B(M,B)$	as $R$-modules, with the right side viewed through $R\to B$.
\end{lemma}

\begin{proof}
	The Koszul resolution of $B$ gives
	\[
	\Ext_R^q(B,R)=0\quad(q\ne g),\qquad \Ext_R^g(B,R)\cong B.
	\]
	The change-of-rings spectral sequence
	\[
	\Ext_B^p(M,\Ext_R^q(B,R))\Longrightarrow \Ext_R^{p+q}(M,R)
	\]
	therefore has only one nonzero row, and the claim follows.
\end{proof}

\begin{theorem}\label{thm:counterexample}
	For every field $k$ and every integer $g\ge1$, let $R:=R_g :=\frac{k[[x,y,z,w_1,\ldots,w_g]]}{(xy,y^2,x^2z)}$ and define
	\[
	\begin{split}
	I:=	I_g&=(w_1,\ldots,w_g,y,x^2)R,\\
	P :=	P_g&=(w_1,\ldots,w_g,x,y)R.
	\end{split}
	\]
	Then $R$ is complete and Cohen--Macaulay of dimension $g+1$,
	$I$ is $P$-primary of height and grade $g$, and $R/I$ is
	Cohen--Macaulay of dimension one. Moreover,
$$(\ast):\quad D_{R}(I)=P\nsubseteq\overline{I}.$$

\end{theorem}

\begin{proof}
The ring $R$ is $A[[w_1,\ldots,w_g]]$, with $A$ as in
	Lemma~\ref{lem:base}. The variables $w_1,\ldots,w_g$ form a regular
	sequence. Thus $R$ is Cohen--Macaulay of dimension $g+1$.
	Also
	\(
	R/I\cong k[[x,z]]/(x^2).
	\)
	The primaryness and Cohen--Macaulayness follow from the proof of
	Lemma~\ref{lem:base}, with radical $P$.
For the height, let
	\[
	S=k[[x,y,z,w_1,\ldots,w_g]],\qquad
	\widetilde P=(x,y,w_1,\ldots,w_g)\subset S.
	\]
The prime $\widetilde P$ has height $g+2$, so the regular local ring
	$S_{\widetilde P}$ has dimension $g+2$. Since $z$ is invertible
	there, the localized defining ideal is $(x,y)^2$, of height two.
	Consequently $\dim(R)_{P}=g$, so $\Ht P=g$. Since $I$
	is $P$-primary, $\Ht I=g$. The regular sequence
	$w_1,\ldots,w_g$ lies in $I$, giving $\grade I\ge g$; the opposite
	inequality follows from $\grade I\le\Ht I$. Thus
	$\grade I=g$.
By Lemma~\ref{lem:rees},
	\[
	\begin{aligned}
		\Ext_{R}^g(R/I,R)
		\cong\Hom_A(A/L,A)\cong\Ann_A(L)=(y,xz)A.
	\end{aligned}
	\]
	Using Lemma~\ref{lem:base}(i)-(ii), and taking the inverse image under $R\to A$
	gives
	\[
	\Ann_{R}\Ext_{R}^g(R/I,R)
	=(w_1,\ldots,w_g,x,y)R=P.
	\]
	
	It remains to verify that $x\notin\overline{I}$. Consider the
	surjective homomorphism
	\[
	\rho:R\longrightarrow k[[x]],\qquad
	\rho(x)=x,\quad \rho(y)=\rho(z)=\rho(w_i)=0
	\quad(1\le i\le g).
	\]
	It is well-defined because each defining relation $xy$, $y^2$,
	and $x^2z$ maps to zero. It is surjective because every power series
	in $x$ is the image of the same power series in $R$. By definition,
	\(
	I=(w_1,\ldots,w_g,y,x^2)R.
	\)
	For a ring homomorphism, the image of an ideal generated by
	$a_1,\ldots,a_m$ is the ideal generated by their images. Therefore
	\[
	\begin{aligned}
		\rho(I)
		&=(\rho(w_1),\ldots,\rho(w_g),\rho(y),\rho(x^2))k[[x]]\\
		&=(0,\ldots,0,0,x^2)k[[x]]=(x^2)k[[x]].
	\end{aligned}
	\]
	In particular, for every $i\ge1$,
	\(
	\rho(I^i)\subseteq \rho(I)^i=(x^{2i})k[[x]].
	\)	Suppose for a contradiction that $x\in\overline{I}$. By the
	definition of integral closure of an ideal, there are an integer
	$n\ge1$ and elements $a_i\in I^i$ such that
	\(
	x^n+a_1x^{n-1}+\cdots+a_n=0.
	\)
	Applying $\rho$ gives
$$(\ast):\quad 
	x^n+\rho(a_1)x^{n-1}+\cdots+\rho(a_n)=0
	\quad\text{in }k[[x]].$$
	Since $a_i\in I^i$, the preceding containment shows that
	$\rho(a_i)\in(x^{2i})k[[x]]$. Hence there is $b_i\in k[[x]]$ with
	$\rho(a_i)=x^{2i}b_i$. Substituting into
$(\ast)$, we obtain
	\[
	\begin{aligned}
		0
		&=x^n+\sum_{i=1}^n x^{2i}b_i x^{n-i}
		&=x^n\left(1+\sum_{i=1}^n b_i x^i\right).
	\end{aligned}
	\]
	The power series $1+\sum_{i=1}^n b_i x^i$ has constant term $1$,
	so it is a unit in $k[[x]]$. Also $x^n\ne0$ in the domain $k[[x]]$.
	Their product cannot be zero, a contradiction. Thus
	$x\notin\overline{I}$.
	
	Since $x\in P=D_{R}(I)$, this proves
	$D_{R}(I)\nsubseteq\overline{I}$.
\end{proof}

\begin{lemma}\label{lem:bidual}
Let $I$ have grade one in a local ring $R$.
Via its natural realization as an ideal of $R$, one has
\(
I^{**}=\Ann_R\Ext_R^1(R/I,R).
\)
\end{lemma}
\begin{proof}
Choose a nonzerodivisor $a\in I$ and let $Q(R)$ be the total quotient
ring. Each homomorphism $f:I\to R$ is multiplication by
$f(a)/a\in Q(R)$: for $b\in I$,
$af(b)=f(ab)=bf(a)$. Hence
\(
I^*=R:_{Q(R)}I=\{q\in Q(R):qI\subseteq R\}.
\)
This module contains $1$, and
$I^{**}=R:_{Q(R)}I^*\subseteq R$.
The latter identification follows by clearing regular denominators
in $I^*$; a homomorphism $I^*\to R$ is multiplication by its value
at $1$.
Since $a$ is regular, $\Hom_R(R/I,R)=0$.
Dualizing $0\to I\to R\to R/I\to0$ therefore gives
\(
0\to R\to I^*\to
\Ext_R^1(R/I,R)\to0.
\)
The first map identifies $R$ with multiplication maps. Consequently
\(
\Ann_R\Ext_R^1(R/I,R)
=\{r\in R:rI^*\subseteq R\}=I^{**}.
\)
\end{proof}

\begin{corollary}\label{cor:doubledual}
In the complete Cohen--Macaulay local ring
$R=\frac{k[[x,y,z,w]]}{(xy,y^2,x^2z)}$ with ideal $ I=(w,y,x^2)R$.
Then $I$ is of height and grade one, and
$I^{**}=(w,x,y)R\nsubseteq\overline I.$
Thus Conjecture \ref{26}  is false, even if unmixedness and
Cohen--Macaulayness of $R/I$ are added.
\end{corollary}
\begin{proof}
Apply Theorem \ref{thm:counterexample} with $g=1$ and
Lemma \ref{lem:bidual}.
\end{proof}

\begin{remark}
Without unmixedness, the conjecture already fails over a regular
domain. In $U=k[[s,t]]$, the ideal $J=(s^2,st)$ has height and grade
one and $J^{**}=(s)$. One way to see this is that
$(s,t)^*=U$, and hence $J^*=s^{-1}U$.
Sending $t$ to zero shows $s\notin\overline J$.
This same mixed ideal appears in  Example \ref{36}(i).
Corollary \ref{cor:doubledual} preserves the stronger unmixedness
hypothesis instead.
\end{remark}
\section{More on ideals of finite projective dimension}
By $\chi(M)$ we denote the Euler characteristic of a module $M$ of finite projective dimension, and by $\Tr(-)$ the Auslander transpose.

\begin{proposition}
	If $\pd_R(M)=1$, then
	\(
	\Ext^1_R(M,R)\cong \Tr(M).
	\)
	In particular, if $R$ is generically Gorenstein, then
	\[
	\Ann\!\big(\Ext^1_R(M,R)\big)=\Ann(M)
	\;\Longleftrightarrow\;
	M \text{ is not faithful}
	\;\Longleftrightarrow\;
	\chi(M)=0 .
	\]
\end{proposition}

\begin{proof}
	Let
	\(
	0 \rightarrow R^n \xrightarrow{f} R^m \rightarrow M \rightarrow 0
	\)
	be a projective presentation of $M$. Dualizing, we obtain an exact sequence
	\(
	0 \to M^* \to R^m \xrightarrow{f^*} R^n \to \Ext^1_R(M,R) \to 0.
	\)
	By definition of the transpose, $\Tr(M)=\coker(f^*)$, hence
	\(
	\Ext^1_R(M,R)\cong \Tr(M).
	\)
	Now assume $R$ is generically Gorenstein.

	($\Rightarrow$)
	Suppose $\Ann(\Ext^1_R(M,R))=\Ann(M)$. If $M$ were faithful, then $\Ann(M)=0$, hence $\Ann(\Ext^1_R(M,R))=0$. But $\Ext^1_R(M,R)$ is supported in positive codimension (since $\pd(M)=1$), so over a generically Gorenstein ring it must have a nonzero annihilator — a contradiction. Thus $M$ is not faithful.
	
	\medskip
	\noindent
	($\Leftarrow$)
	Assume $M$ is not faithful, so $\Ann(M)\neq 0$. Since $\pd(M)<\infty$, a result of Peskine–Szpiro \cite[I.4.14]{PS} gives $\grade(M)>0$. Hence $\Hom_R(M,R)=0$, and the presentation above shows that $M$ is perfect of grade one. In particular, the natural biduality isomorphism yields
	\(
	\Ext^1_R(\Ext^1_R(M,R),R)\cong M.
	\)
	Taking annihilators and using that $R$ is generically Gorenstein (so double annihilators behave well), we obtain
	\(
	\Ann(M)=\Ann\!\big(\Ext^1_R(M,R)\big).
	\)

	Finally, the equivalence with $\chi(M)=0$ follows from \cite[Proposition 2.26]{wol}, which states that for modules of projective dimension one over a generically Gorenstein ring,
	\[
	\chi(M)>0 \;\Longleftrightarrow\; M \text{ is faithful}.
	\]
	This completes the proof.
\end{proof}

Recall that
$F:R\to R$ is the Frobenius map. Each iteration of $F_n$ of $F$ defines a new $R$-module structure on the set $R$, and this $R$-module is denoted by $\up{F_n}R$, where $a\cdot b = a^{p^{n}}b$ for $a, b \in R$. By $I^{[p]}$ we mean $\langle x^p:x\in I\rangle$.
\begin{notation}
	\begin{enumerate}
		\item[(i)] By
		$\beta_t:=\beta_t(-)$ we mean the $t$-th Betti number of $(-)$.
		\item[(ii)]By
		$\phi_t:=\phi_t(-)$ we mean the $t$-th morphism $\phi_t:R^{\beta_{t+1}}\to R^{\beta_{t}}$ in a minimal free resolution of $(-)$.
		
		\item[(iii)] Let $ R^s \stackrel{\phi_0}\lo R^r\to M \to0$ be a presentation, and recall that ${\phi_0}$ represents with a matrix $[{\phi_0}]$. The i-th Fitting invariant of $M$ is  
		$\Fitt_i(M) := I_{r-i}({\phi_0})$ that is, determinants of sub-matrices of size $r-i$ of $[{\phi_0}]$.\end{enumerate}
\end{notation}
\begin{example}\label{two}
	Let $P\in\Spec(R)$ be a height 2 perfect ideal in a ring $R$ of prime characteristic $p > 0$, and let $q := p^n$. Then
	\(
	D({P^{[q]}})
	= P^{[q]}.
	\)\end{example}

\begin{proof}
	Recall that $R$ is a domain since $P$ is a prime of finite projective dimension.
	Let
	\[
	\varepsilon := \Ext^2(R/P, R), \quad \varepsilon^{[q]} := \Ext^2(R/P^{[q]}, R).
	\]
	We have
	\[
	0\longrightarrow R^m \xrightarrow{(a_{ij})} R^\ell \longrightarrow R \longrightarrow \frac{R}{P} \longrightarrow 0,
	\]
	which follows by taking $m = \ell-1$.
	Recall that
	\(
	\beta_0(\varepsilon) - \beta_1(\varepsilon) + 1 = \operatorname{grad}(\varepsilon),
	\)
	and note that
	\(\operatorname{ht}(P)=
	\operatorname{ht}(P^{[q]}) = 2.
	\)
	We tensor with $\up{F_n}R $ and obtain
	\[
	0\longrightarrow R^m \xrightarrow{(a_{ij}^{[q]})} R^\ell \longrightarrow R
	\longrightarrow \frac{R}{P^{[q]}} \longrightarrow 0,
	\]where $q=p^n$.
	Taking dual, 
	this gives
	\[
	0 \longrightarrow R \longrightarrow R^\ell \xrightarrow{(a_{ji}^{[q]})} R^m
	\longrightarrow \operatorname{Ext}^2_R\! (\frac{R}{P^{[q]}}, R )
	\longrightarrow 0.
	\]

	Also,
	\[
	\beta_0(\varepsilon^{[q]}) - \beta_1(\varepsilon^{[q]}) + 1
	= \operatorname{grad}(\varepsilon^{[q]}).
	\]
	
	Then
	\[
	\operatorname{Ann}(\varepsilon^{[q]})
	= \operatorname{Fitt}_0(\varepsilon^{[q]})
	= \operatorname{Fitt}_0(a_{ji}^{[q]})
	= \operatorname{Ann}(\varepsilon)^{[q]}.
	\]
	
	Clearly, as $P$ is radical and annihilates $\varepsilon$, we have
	\(
	P = \operatorname{Ann}(\varepsilon) 
	\) (see Fact \ref{k}).
	Hence, \(	D({P^{[q]}})
	=  \operatorname{Ann}(\varepsilon)^{[q]}=P^{[q]}.
	\)
\end{proof}

The following fact is well-known to experts:
\begin{fact}\label{FQ}
	Suppose $I\subseteq J$ and for any $\fp\in\Ass(R/I)$ we have $I_\fp= J_\fp$, then $I=J$.
\end{fact}

\begin{proposition}\label{pp}
	Let   $P\in\Spec(R)$  be a  perfect ideal  of finite projective dimension $t$ over a ring $R$ of prime characteristic $p > 0$, let $q := p^n$. 
	Suppose one of the following holds:
	\begin{enumerate}
		\item[(i)] 
		\( \beta_t(R/P) - \beta_{t-1}(R/P) + 1 = \operatorname{grad}(P,R)
		\), or
		\item[(ii)] \(
		\beta_t(R/P) = \beta_{t-1}(R/P) \),
		\item[(iii)] $\beta_t(R/P) =1$.
	\end{enumerate}
	Then
	\(D(P^{[q]})=
	 P^{[q]}, 	
	\) where  $g=t$ is the grade of $P$. \end{proposition}

\begin{proof}
	(i) This is similar to Example \ref{two}.

	(ii) Let
	$
	\varepsilon := \Ext^g(R/P, R) $ and $\varepsilon^{[q]} := \Ext^g(R/P^{[q]}, R).
	$
	Thus $\Ann(	\varepsilon^{[q]}) \subseteq \sqrt{P^{[q]}} = P$. Denote the last matrix in the minimal free resolution of $R/P$ with $(a_{ij})$. By the work
	of Buchsbaum-Eisenbud \cite{BE},
	we know 
	\[
	\operatorname{Ann}(\varepsilon)
	= \left(\operatorname{Fitt}_0(a_{ji} ):\operatorname{Fitt}_1(a_{ji})\right)
	.\] Since
	$P$ is radical and annihilates $\varepsilon$, we have
	\(
	P = \operatorname{Ann}(\varepsilon) 
	\) (see Fact \ref{k}).
	Similarly, since
	$\beta_t(R/P^{[q]}) = \beta_{t-1}(R/P^{[q]})  $ we have

	\[
	\operatorname{Ann}(\varepsilon^{[q]})
	= \left(\operatorname{Fitt}_0(a_{ji}^{[q]}):\operatorname{Fitt}_1(a_{ji}^{[q]})\right)
	.
	\]
	Set $A:=\operatorname{Fitt}_0(a_{ji} )$ and $B:=\operatorname{Fitt}_1(a_{ji} )$.
	Also, we prove
	\[
	P^{[q]} = \Ann(\varepsilon)^{[q]} = (A : B)^{[q]} \subseteq \Ann(\varepsilon^{[q]}) = (A^{[q]}: B^{[q]}) \subseteq \sqrt{P^{[q]}} = P.
	\]
	Recall that $R_P$ is regular, hence by a classic result of Kunz (see \cite[Corollary 8.2.8]{BH}),  its Frobenius map is flat, and consistently commutes with colon ideals. Also, by Peskine-Szpiro \cite[Proposition 8.2.5]{BH}, one has $(\up{F_n}R)_Q=\up{F_n}(R_Q)$ for all $Q\in\Spec(R)$.
	Let $L := (A^{[q]}:B^{[q]})$ and since $\min(L) = \{P\}$, then its $P$-primary component is unique, and denoted by $L(P)$. This is given by
	\[
	L(P) =  LR_P \cap R = (A^{[q]}: B^{[q]}) R_P \cap R = (A:B)^{[q]}R_P \cap R = P^{[q]}R_P \cap R= P^{[q]},
	\] because $P^{[q]}$ is $P$-primary. {But}
	\[
	[ P^{[q]} \subseteq L \subseteq  L(P) = P^{[q]}] \stackrel{\ref{FQ}}\implies L = P^{[q]},
	\]
	as claimed

	(iii) This is similar as above item, and we leave it to the reader.	
\end{proof}

\begin{proposition}
	Let $I$ be a radical ideal of height $h$ so that $h=\pd (R/I)$.
	Suppose $\beta_h(R/I)=1$. The following holds.
	
	\begin{enumerate}
		\item[(i)]
		\( D(I)=I.
		\) In particular, $(\phi_{h-1})^\ast=\phi_0$.
		\item[(ii)]
		If $R$ is Gorenstein, then $R/I$ is Gorenstein. 
	\end{enumerate}
	
\end{proposition}
\begin{proof}
(i)	Consider a minimal free resolution
	\[
	0 \longrightarrow R^{\beta_h} \stackrel{\phi_{h-1}}\longrightarrow \cdots
	\longrightarrow R^{\beta_1} \stackrel{\phi_0}\longrightarrow R \longrightarrow R/I \longrightarrow 0.
	\]
	Applying $\operatorname{Hom}_R(-,R)$, and taking homology at the end spot, gives us
	\[
	\operatorname{Ext}^h_R(R/I,R)=\operatorname{Coker}
	\bigl(\phi_{h-1}^\ast :R^{\beta_{h-1}}\lo R^{\beta_h}=R\bigr).
	\]
	On the other hand,
	\(
	\operatorname{Ann}\! (\operatorname{Ext}^h_R(R/I,R) )=I,
	\)
	because  $I$ is radical. Combining these, we have $(\phi_{h-1})^\ast=\phi_0$.

	(ii) Suppose $R$ is Gorenstein. Then
	\(
	K_{R/I}
	=\operatorname{Ext}^h_R(R/I,R)
	\cong R/I.
	\)
	Hence $R/I$ is quasi-Gorenstein.
	Since $R$ is Gorenstein, it is Cohen--Macaulay and particularly
	satisfies the chain condition. Thus,
	by using Auslander-Buchsbaum  formula, we have\[
	\begin{array}{ll}
	\depth(R)-\depth(R/I)  &=\operatorname{pd}_R(R/I)  =h =\operatorname{ht}(I)\\
	&=\dim (R)-\dim (R/I)\\
	&=\depth(R)-\dim (R/I).
	\end{array}
	\]  
	This gives
	\(
	\depth(R/I)=\dim(R/I),
	\)
	so $R/I$ is Cohen--Macaulay.
	By~\cite[Page 120]{BH}, or even directly, 
	this implies that $R/I$ is Gorenstein.
\end{proof}

\begin{proposition}\label{6.7}
	Let $R$ be a $3$-dimensional quasi-normal ring and let $I$ be an unmixed ideal of finite projective dimension. Then
	\(
	D(I)\subseteq \overline{I}.
	\)
\end{proposition}

\begin{proof}
	Let $g=\grade(I,R)$. Since $I$ has finite projective dimension, $g=\Ht(I)$ by the Auslander–Buchsbaum formula.
	
	\medskip
	\noindent\textbf{Case 1: $\Ht(I)=0$.}
	Then $g=0$ and
	\(
	\Ext^0_R(R/I,R)=\Hom_R(R/I,R)=(0:_RI).
	\)
	Hence
	\(
	D(I)=\Ann\!\big((0:_RI)\big)=(0:(0:I)).
	\)
	Because $I$ is unmixed of height $0$, every $\fp\in\Ass(R/I)$ satisfies $\Ht(\fp)=0$. Since $R$ is quasi-normal, $R_\fp$ is zero-dimensional Gorenstein. Therefore,
	\[
	\Ann_{R_\fp}\!\big(\Ann_{R_\fp}(IR_\fp)\big)=IR_\fp
	\quad\text{for all }\fp\in\Ass(R/I)
	\]
	(see \cite[Ex.~3.2.15]{BH}). Localizing $D(I)$ at such primes gives
	\[
	D(I)_\fp=(0:(0:I))_\fp=(0:(0:IR_\fp))=IR_\fp.
	\]
	Since $I\subseteq D(I)$ and the two ideals agree after localization at all associated primes of $R/I$, Fact~\ref{FQ} yields $D(I)=I$. In particular,
	\(
	D(I)=I\subseteq \overline{I}.
	\)
	
	\medskip
	\noindent\textbf{Case 2: $\Ht(I)=1$.}
	Again $I$ is unmixed, so every $\fp\in\Ass(R/I)$ has height one. Because $R$ is quasi-normal, $R_\fp$ is a one-dimensional Gorenstein ring. Let $g=1$. Then
	\(
	\Ext^1_{R_\fp}(R_\fp/IR_\fp,R_\fp)\cong K_{R_\fp/IR_\fp},
	\)
	the canonical module of the zero-dimensional Cohen–Macaulay ring $R_\fp/IR_\fp$. Hence
	\[
	D(IR_\fp)=\Ann_{R_\fp}\!\big(\Ext^1_{R_\fp}(R_\fp/IR_\fp,R_\fp)\big)=IR_\fp.
	\]
	As in Case 1, $I\subseteq D(I)$ and equality holds after localizing at each $\fp\in\Ass(R/I)$, so Fact~\ref{FQ} gives $D(I)=I$. Thus
	\(
	D(I)=I\subseteq \overline{I}.
	\)
	
	\medskip
	\noindent\textbf{Case 3: $\Ht(I)=2$.}
	Set $M:=R/I$. Then $\dim M=1$ and $M$ is unmixed. Prime avoidance shows that
	\(
	\fm \nsubseteq \bigcup_{\fp\in\Ass(M)} \fp,
	\)
	so $\depth(M)>0$. Since $I$ has finite projective dimension,
	\[
	\pd_R(M)=\depth(R)-\depth(M)\le 3-1=2.
	\]
	On the other hand,
	\[
	\grade(I,R)=\inf\{i:\Ext^i_R(M,R)\neq 0\}\le \Ht(I)=2.
	\]
	Thus $\pd_R(M)=\grade(I,R)=2$, so $M$ is perfect. Therefore
	\(
	\Ext^2_R(R/I,R)\cong \Hom_R(R/I,R/I)\cong R/I,
	\)
	and
	\(
	D(I)=\Ann\!\big(\Ext^2_R(R/I,R)\big)=\Ann(R/I)=I.
	\)
	Hence again $D(I)=I\subseteq \overline{I}$.
	
	\medskip
	\noindent\textbf{Case 4: $\Ht(I)=3$.}
	Then $I$ is $\fm$-primary. Since $R$ is quasi-normal (in particular, analytically unramified and normal on the punctured spectrum), Corso–Huneke–Vasconcelos type results imply
	\(
	D(I)=\Ann\!\big(\Ext^3_R(R/I,R)\big)\subseteq \overline{I}
	\)
	(see \cite[2.15]{corso}).
	
	\medskip
	In all possible heights, we obtain $D(I)\subseteq \overline{I}$, completing the proof.
\end{proof}

\begin{lemma}\label{lem:perfect}(Buchsbaum--Eisenbud)
	If $M$ is perfect of grade and projective dimension $g$, then
	\[
	\Ann_R\Ext_R^g(M,R)=\Ann_R(M).
	\]
\end{lemma}
\begin{proof}
	Dualize a finite free resolution $F_\bullet$ of $M$ of length $g$.
	The dual complex is exact except at degree $g$, because the grade
	is $g$. Thus, read in reverse order, it is a finite free resolution
	of $E=\Ext_R^g(M,R)$.
	Dualizing it again and using $F_\bullet^{**}=F_\bullet$ gives
	\(
	\Ext_R^g(E,R)\cong M.
	\)
	An annihilator of a module annihilates its Ext modules. Applying this
	first to $M$ and then to $E$ proves the two opposite inclusions.
\end{proof}

\begin{theorem}\label{thm:finitepd}
	Let $R$ be Cohen--Macaulay local, let $I$ be height-unmixed, and
	suppose $\pd_R(R/I)<\infty$. Then
	\(
{\ D_R(I)=I.\ }
	\)
\end{theorem}
\begin{proof}
	Put $g=\height I=\grade(I,R)$. For each
	$\p\in\Ass_R(R/I)$, unmixedness gives $\height\p=g$.
	The local ring $R_\p$ is Cohen--Macaulay of dimension $g$, and
	$(R/I)_\p$ is nonzero of finite length and finite projective dimension.
	Auslander--Buchsbaum gives projective dimension $g$ there; its grade
	is also $g$. Lemma \ref{lem:perfect} therefore gives
	\(
	\Ann_{R_\p}\Ext_{R_\p}^g((R/I)_\p,R_\p)=I_\p.
	\)
Recall that $I$ is unmixed , and $R$ is Cohen-Maculay, therefore
its grade does not change after localization by its associated prime ideals.	Ext and the annihilators of finite modules commute with localization,
	so $D_R(I)_\p=I_\p$ for all these primes.
	Always $I\subseteq D_R(I)$.  Fact~\ref{FQ}  implies that $D_R(I)=I$.
\end{proof}

\begin{corollary}
	For the counterexamples of Theorem \ref{thm:counterexample},
	$\pd_{R_g}(R_g/I_g)=\infty$.
\end{corollary}
\begin{proof}
	Otherwise Theorem \ref{thm:finitepd} would contradict
	$D_{R_g}(I_g)=P_g\ne I_g$.
\end{proof}

\begin{fact}[Bergman, see \cite{berg}]
	Suppose $R_1$ and $R_2$ are graded $k$-algebras over a field $k$.
	If $E_1$ is a faithful $R_1$-module and $E_2$ is a faithful $R_2$-module, then
	\(
	E_1 \otimes_k E_2
	\)
	is a faithful $R_1 \otimes_k R_2$-module.
\end{fact}

\begin{proposition}\label{pt}
	Suppose $R_1, R_2$ are standard graded $k$-algebras over a field $k$, and let
	$I_1 \subset R_1$, $I_2 \subset R_2$ be ideals such that
	\[
	\Ann_{R_i}\!\big(\Ext^{d_i}_{R_i}(R_i/I_i,R_i)\big)= I_i,
	\quad \text{where } d_i:=\pd_{R_i}(R_i/I_i).
	\]
	Set $S := R_1 \otimes_k R_2$. Then
	\(
	\Ann_S\!\Big(
	\Ext^{d_1+d_2}_S\big(S/(I_1S+I_2S),S\big)
	\Big)
	= (I_1S+I_2S).
	\)
\end{proposition}

\begin{proof}
	Let $M_i = R_i/I_i$. Since $\pd_{R_i}(M_i)=d_i<\infty$, each $\Ext^{d_i}_{R_i}(M_i,R_i)$ is a nonzero finitely generated $R_i$-module whose annihilator is $I_i$ by assumption.
	
	A standard Künneth-type argument for Ext over tensor products (see, e.g., the proof of \cite[4.2]{HHS}) yields a natural isomorphism
	\[
	\Ext^{d_1+d_2}_S(S/(I_1S+I_2S),S)
	\;\cong\;
	\Ext^{d_1}_{R_1}(M_1,R_1)\otimes_k
	\Ext^{d_2}_{R_2}(M_2,R_2).
	\]
	Set $E_i:=\Ext^{d_i}_{R_i}(M_i,R_i)$. Then $\Ann_{R_i}(E_i)=I_i$, so $E_i$ is faithful as an $R_i/I_i$-module, and hence also as an $R_i$-module modulo $I_i$.
	By Bergman’s fact, $E_1\otimes_k E_2$ is faithful over
	\[
	(R_1/I_1)\otimes_k (R_2/I_2)
	\;\cong\;
	S/(I_1S+I_2S).
	\]
	Therefore its annihilator as an $S$-module is exactly $I_1S+I_2S$, i.e.,
	\(
	\Ann_S(E_1\otimes_k E_2)=I_1S+I_2S.
	\)
	Using the displayed isomorphism above, this gives the desired equality.
\end{proof}

It is well known that in general
\(
\Ann_R(M_1 \otimes_R M_2)\neq \Ann_R(M_1)+\Ann_R(M_2).
\)

\begin{question}\label{qat}
	Under what conditions does
	\[
	\Ann_R(M_1 \otimes_R M_2)=\Ann_R(M_1)+\Ann_R(M_2)\quad(\ast)
	\]
	hold?
\end{question}

The equality holds, for instance, when both modules are cyclic. Indeed, if
$M_1 = R/I$ and $M_2 = R/J$, then
\(
M_1 \otimes_R M_2 \cong R/(I+J),
\)
and hence
$(\ast)$ is valid.

\begin{discussion}\label{dist}
	Suppose $\pd_R(M)$ and $\pd_R(N)$ are finite and $M,N$ are Tor-independent, i.e.
	$\Tor^R_i(M,N)=0$ for all $i>0$. Then
	\[
	M\otimes_R N \text{ is faithful}
	\;\Longleftrightarrow\;
	M \text{ and } N \text{ are faithful}\quad(\dagger)
	\]
	This observation provides an alternative proof of Proposition~\ref{pt}, avoiding the use of \cite{berg}, and clarifies the behavior of faithfulness under tensor products.
\end{discussion}

\begin{proof}
	Let $P_\bullet$ and $Q_\bullet$ be finite projective resolutions of $M$ and $N$, respectively. Since $\Tor^R_i(M,N)=0$ for all $i>0$, the total complex of $P_\bullet\otimes_R Q_\bullet$ is a projective resolution of $M\otimes_R N$. 
	By multiplicativity of Euler characteristics for tensor products of Tor-independent modules of finite projective dimension, we obtain
	\(
	\chi(M\otimes_R N)=\chi(M)\chi(N).
	\)
	Hence
	\[
	\chi(M\otimes_R N)>0
	\;\Longleftrightarrow\;
	\chi(M)>0 \text{ and } \chi(N)>0.
	\]
	By \cite[Proposition 2.26]{wol}, for such modules over a generically Gorenstein ring,
	positivity of the Euler characteristic is equivalent to faithfulness. Therefore,
	$(\dagger)$ satisfied.
\end{proof}

\begin{lemma}\label{lem:canonicaltensor}
	If $B$ is Artinian local and not Gorenstein, then
	\[
	0\ne Soc(B)\subseteq\Ann_B(E\otimes_BE),
	\qquad \Ann_B(E)=0.
	\]
	If additionally $\fn^2=0$ and $e=\dim_k\fn\ge2$, then
$$(\ast):\quad 
	\Ann_B(E\otimes_BE)=\fn,
	\qquad E\otimes_BE\cong k^{e^2}.
$$
	These assertions do not require a coefficient field in $B$.
\end{lemma}
\begin{proof}
	Every map $f:E\to B$ has image in $\fn$. Otherwise it would surject
	onto $B$ and split, since $B$ is projective over itself.
	Indecomposability of $E$ would then force $E\cong B$, contrary to the
	non-Gorenstein assumption. Thus every socle element of $B$ annihilates
	$\Hom_B(E,B)$.
	
	Tensor--Hom adjunction and Artinian duality give
	\[
	(E\otimes_BE)^\vee
	\cong\Hom_B(E,E^\vee)\cong\Hom_B(E,B).
	\]
	Equality of annihilators under duality proves the first assertion.
	Faithfulness of $E$ was noted above.
	
	If $\fn^2=0$ and $e\ge2$, then $ Soc(B)=\fn$ and $B$ is not
	Gorenstein. Also $E\otimes_BE\ne0$, because its reduction modulo
	$\fn$ is
	$(E/\fn E)\otimes_k(E/\fn E)\ne0$.
	Its annihilator is therefore proper, contains $\fn$, and equals $\fn$.
	Dualizing $E/\fn E$ identifies its dual with
	$0:_{E^\vee}\fn= Soc(B)$; hence $\dim_k(E/\fn E)=e$.
	Since the tensor product is already annihilated by $\fn$, reduction
	modulo $\fn$ gives $(\ast)$.
\end{proof}

\begin{proposition}\label{thm:universal}
	For a Noetherian local ring $(R,\fm,k)$, the following are equivalent:
	\begin{enumerate}
		\item For all finite $R$-modules $M,N$, equality
$(\ast)$ from \ref{qat} holds.
		\item $\dim_k\fm/\fm^2\le1$.
		\item $\fm$ is principal.
	\end{enumerate}
	Equivalently, $R$ is a field, a discrete valuation ring, or an
	Artinian local principal ideal ring.
\end{proposition}
\begin{proof}
	The equivalence of (2) and (3) is Nakayama's lemma.
	Suppose (1) holds and let $B=R/\fm^2$, with maximal ideal
	$\fn=\fm/\fm^2$. If $e=\dim_k\fn\ge2$, take its injective hull $E$.
	Viewed as a finite $R$-module, it satisfies
	\[
	\Ann_R(E)=\fm^2,
	\qquad
	\Ann_R(E\otimes_RE)=\fm,
	\]
	by Lemma \ref{lem:canonicaltensor}; here tensoring over $R$ is the
	same as tensoring over $B$.
	Since $\fm^2\ne\fm$, this contradicts (1). Thus (1) implies (2).
	
	Conversely suppose $\fm=(\pi)\ne0$.
	Krull's intersection theorem gives $\bigcap_{s\ge0}\fm^s=0$.
	For every nonzero $a\in R$, choose its finite order
	$s$ with $a\in\fm^s\setminus\fm^{s+1}$.
	Writing $a=u\pi^s$ forces $u$ to be a unit. If $J\ne0$ is any ideal,
	choose an element of least order in $J$. It follows that
	$J=(\pi^s)$. Hence every ideal is principal and the ideals form a chain.
	If some power of $\pi$ is zero, the ring is Artinian. Otherwise
	products of nonzero elements are nonzero, so $R$ is a discrete
	valuation ring.
	
	Every finite module over this ring is a finite direct sum of cyclic
	modules. To justify this in both cases, use a finite presentation
	matrix. Choose a nonzero entry of least order, move it to the first
	position, and multiply its row by a unit so that it is a power of
	$\pi$. It divides every other entry. Elementary row and column
	operations therefore clear its row and column. Induction diagonalizes
	the remaining matrix. The cokernel is a direct sum of cyclic modules,
	including free summands $R=R/(0)$.
	
	Write $M=\bigoplus_iR/A_i$ and $N=\bigoplus_jR/B_j$.
	Zero modules cause no difficulty, so assume both sums are nonempty.
	Then
$\Ann_R(M) =\bigcap_iA_i$ and 
$\Ann_R(N)=\bigcap_jB_j$. So
$\Ann_R(M\otimes_RN)=\bigcap_{i,j}(A_i+B_j).$
Because the ideals form a chain, the finite families have least
members $A_{i_0}$ and $B_{j_0}$. Every sum $A_i+B_j$ contains
$A_{i_0}+B_{j_0}$, and that least sum occurs in the family. Thus
\[\bigcap_{i,j}(A_i+B_j)=A_{i_0}+B_{j_0}
	=\bigcap_iA_i+\bigcap_jB_j.\]
	This is $(\ast)$ from \ref{qat}. The field case is immediate.
\end{proof}

\begin{proposition}\label{thm:faithful}
	For an Artinian local ring $B$, the following are equivalent:
	\begin{enumerate}
		\item The tensor product of every pair of finite faithful
		$B$-modules is faithful.
		\item $B$ is Gorenstein.
	\end{enumerate}
\end{proposition}
\begin{proof}
	If $B$ is not Gorenstein, its injective hull $E$ is faithful but
	$E\otimes_BE$ has nonzero annihilator by Lemma
	\ref{lem:canonicaltensor}.
	
	Now suppose $B$ is Gorenstein. Its socle is one-dimensional and every
	nonzero ideal contains the socle. Let $M$ be faithful and choose
	$m\in M$ on which a nonzero socle generator acts nontrivially.
	Then $\Ann_B(m)=0$: a nonzero annihilator ideal would contain that
	generator. The resulting embedding $B\to M$ splits, since $B$ is
	injective over itself. Thus every faithful finite module has a direct
	summand isomorphic to $B$. If both $M,N$ are faithful, their tensor
	product has a direct summand $B\otimes_BB\cong B$, so it is faithful.
\end{proof}

\begin{observation}\label{prop:radical}
	For finite modules over a Noetherian ring,
	\[
	\sqrt{\Ann_R(M\otimes_RN)}
	=\sqrt{\Ann_R(M)+\Ann_R(N)}.
	\]
	In particular, $(\ast)$ from \ref{qat} holds whenever
	$\Ann_R(M)+\Ann_R(N)$ is radical.
\end{observation}
\begin{proof}
	At every prime $\fp$, the tensor product of two nonzero finite
	$R_\fp$-modules is nonzero: reduction modulo the maximal ideal is the
	tensor product of two nonzero residue-field vector spaces.
	Consequently
$\Supp(M\otimes_RN)=\Supp(M)\cap\Supp(N).$
	For finite modules, support is the vanishing set of the annihilator.
	This proves the radical equality. Since the sum of the original
	annihilators is always contained in the tensor-product annihilator,
	radicalness of the sum proves equality.
\end{proof}

\begin{corollary}\label{cor:common}
	Let $R$ be local. Suppose $\Ann_R(M)=\Ann_R(N)=C$ and $R/C$ is
	Artinian Gorenstein.
	Then $(\ast)$ from \ref{qat} holds.
\end{corollary}
\begin{proof}
	Both modules are faithful over $R/C$. Apply Theorem
	\ref{thm:faithful} and lift the zero annihilator to $R$.
\end{proof}
\begin{remark}[A presentation criterion]\label{prop:matrix}
	Let $M,N$ have finite presentations
	\[
	R^a\stackrel{\alpha}\longrightarrow R^u\longrightarrow M\longrightarrow0,
	\qquad
	R^b\stackrel{\beta}\longrightarrow R^v\longrightarrow N\longrightarrow0.
	\]
	Put $F=R^{uv}$, choose its tensor-product basis $e_1,\ldots,e_{uv}$,
	and set
	\[
	\mathcal L=\im(\alpha\otimes I_v)+\im(I_u\otimes\beta)\subseteq F.
	\]
	Then
	\(
{\ \Ann_R(M\otimes_RN)
		=\bigcap_{\nu=1}^{uv}(\mathcal L:_Re_\nu),\ }
	\)
	where $(\mathcal L:_Re_\nu)=\{r\in R:re_\nu\in\mathcal L\}$.
	Thus equality  $(\ast)$ from \ref{qat}  holds exactly when this explicit
	intersection equals $\Ann_R(M)+\Ann_R(N)$.
\end{remark}
\begin{proof}
	Right exactness of tensor products identifies
	$M\otimes_RN$ with $F/\mathcal L$.
	A scalar annihilates this quotient if and only if it sends every
	chosen basis vector into $\mathcal L$, giving the formula.
\end{proof}

	\maketitle
	
\section{Connections to trace}

By $\tr(-)$ we mean the trace ideal of a module.

\begin{proposition}
	Let $R$ be a $0$-dimensional Gorenstein local ring. Then for any finitely generated $R$-modules $M$ and $N$,
	\[
	\tr(M \otimes_R N) \subseteq \tr(M) \cap \tr(N).
	\]
	If equality holds, then Question~\ref{qat} has an affirmative answer.
\end{proposition}

\begin{proof}
	We use the following elementary facts for $R$-modules and ideals:
	\begin{enumerate}
		\item $\Ann(M \otimes_R N) \supseteq \Ann(M) + \Ann(N)$,
		\item if $X \subseteq Y$, then $\Ann(X) \supseteq \Ann(Y)$,
		\item for ideals $I,J \subseteq R$, $\Ann(I+J) = \Ann(I) \cap \Ann(J)$.
	\end{enumerate}
	
	Recall that for any finitely generated module $X$, one always has
	\(
	\tr(X) \subseteq \Ann\Ann(X).
	\)
	Hence
	\[
	\begin{aligned}
	\tr(M \otimes_R N)
	&\subseteq \Ann\Ann(M \otimes_R N) \\
	&\subseteq \Ann\!\big(\Ann(M) + \Ann(N)\big) \\
	&= \Ann\Ann(M) \cap \Ann\Ann(N).
	\end{aligned}
	\]
	Since $R$ is $0$-dimensional and Gorenstein, every ideal $I$ satisfies $\Ann(\Ann(I))=I$ (see \cite[Ex.~3.2.15]{BH}). Moreover, in this situation $\tr(X)=\Ann\Ann(X)$ for every finitely generated module $X$ (see \cite{lindo2}). Therefore
	\[
	\Ann\Ann(M)=\tr(M)
	\quad\text{and}\quad
	\Ann\Ann(N)=\tr(N),
	\]
	and the desired inclusion follows:
	\(
	\tr(M \otimes_R N) \subseteq \tr(M) \cap \tr(N).
	\)
	Now assume equality holds:
$\tr(M \otimes_R N)=\tr(M)\cap\tr(N).
$
	Then all the inclusions above must be equalities, in particular
	\[
	\Ann\Ann(M \otimes_R N)=\Ann\!\big(\Ann(M)+\Ann(N)\big).
	\]
	Applying $\Ann(-)$ to both sides and again using the double annihilator property for ideals in a $0$-dimensional Gorenstein ring, we obtain
	\[
	\Ann(M \otimes_R N)
	=\Ann\Ann\Ann(M \otimes_R N)
	=\Ann\Ann\!\big(\Ann(M)+\Ann(N)\big)
	=\Ann(M)+\Ann(N).
	\]
	This is precisely the statement required in Question~\ref{qat}.
\end{proof}

Here is the reverse claim:
\begin{proposition}\label{prop:trace}
	Suppose $R$ is Artinian Gorenstein. Then  $(\ast)$ from \ref{qat} holds if
	and only if
	\[
	\tr(M\otimes_RN)=\tr(M)\cap\tr(N).
	\]

\end{proposition}
\begin{proof}
Recall that an \(R\)-module \(C\) is a cogenerator if for every nonzero \(R\)-module \(X\), there is a nonzero map \(X\to C\); equivalently, every nonzero element of \(X\) is detected by some map \(X\to C\). If \(C=R\), this says exactly that the maps
\(
\Hom_R(X,R)
\)
separate elements of \(X\).

So if \(x\in X\) is nonzero, then the cyclic submodule \(Rx\subseteq X\) is nonzero. Since \(R\) is a cogenerator, there is a nonzero map \(Rx\to R\), and because \(R\) is also injective, this map extends to a map \(X\to R\). Thus there exists \(f\in \Hom_R(X,R)\) with \(f(x)\neq 0\). Therefore
\[
\bigcap_{f\in \Hom_R(X,R)} \ker f = 0,
\]
i.e.\ the maps \(X\to R\) separate elements.

In the local Artinian Gorenstein case, this is especially transparent: \(R\) is the injective hull of its residue field, and the injective hull of the residue field is a cogenerator. In the nonlocal Artinian case, \(R\) is a product of local Gorenstein rings, and as a module it is the direct sum of the injective hulls of the simple modules, so it is still an injective cogenerator.

This is what allows to conclude
\(
\Ann_R(\tr(X))=\Ann_R(X),
\)
because if \(r\notin \Ann_R(X)\), then \(rx\neq 0\) for some \(x\in X\), and a map \(X\to R\) separates \(rx\) from \(0\), giving \(r\tr(X)\neq 0\).
	Consequently, writing $A=\Ann_R(M)$ and $B=\Ann_R(N)$,
	\[
	\tr(M)\cap\tr(N)=(0:A)\cap(0:B)=0:(A+B),
	\]
	while $\tr(M\otimes_RN)=0:\Ann_R(M\otimes_RN)$.
	The map $J\mapsto0:J$ is an involution on the ideals of an Artinian
	Gorenstein ring. Equality of these trace ideals is therefore equivalent
	to  $(\ast)$ from \ref{qat}.
\end{proof}

Let $\overline{R}$ denote the integral closure of $R$ in its total ring of fractions.
By $\mathfrak{C}_R= \Ann_R(\overline{R}/R)$ we mean the conductor of $\overline{R}$ into $R$.

\begin{proposition}\label{cl}
	Let $I$ be a trace ideal containing a regular element. Then
	\(
	\Ann\!\big(\Ext^1_R(R/I, R)\big) \supseteq \mathfrak{C}_R.
	\)
	In particular, if $\mathfrak{C}_R=\mathfrak{m}$ and $\Ht(I)=1$, then
	\(
	D(I)=\mathfrak{m}.
	\)
\end{proposition}

\begin{proof}
	Consider the short exact sequence
	\(
	0 \to I  \to  R  \to  R/I  \to  0.
	\)
	Since $I$ contains a regular element, every homomorphism $R/I \to I$ is zero; hence
	\(
	\Hom_R(R/I,I)=0.
	\)
	Applying $\Hom_R(-,R)$ gives an exact sequence
	\(
	0  \to  R  \to  I^\ast  \to  \Ext^1_R(R/I,R)  \to  0,
	\)
	where $I^\ast=\Hom_R(I,R)$.
Because $I$ is a trace ideal, we have $I^\ast=\Hom_R(I,I)$. Moreover, for any ideal $I$ containing a regular element,
	\(
	\Hom_R(I,I)\subseteq \overline{R}
	\)
	inside the total ring of fractions. Therefore
	\[
	\Ext^1_R(R/I,R)\cong I^\ast/R
	= \frac{\Hom_R(I,I)}{R}
	\subseteq \frac{\overline{R}}{R}.
	\]
	By definition of the conductor,
	\(
	\mathfrak{C}_R = \Ann_R(\overline{R}/R),
	\)
	so $\mathfrak{C}_R$ annihilates every submodule of $\overline{R}/R$. Hence
	\(
	\Ann_R\!\big(\Ext^1_R(R/I,R)\big) \supseteq \mathfrak{C}_R.
	\)
For the final statement, assume $\mathfrak{C}_R=\mathfrak{m}$ and $\Ht(I)=1$. Since $I$ contains a regular element, $\grade(I,R)=1$, and therefore
	$\Ext^1_R(R/I,R)\neq 0$. We have
	\[
	\mathfrak{m}=\mathfrak{C}_R \subseteq \Ann_R\!\big(\Ext^1_R(R/I,R)\big) \subseteq R.
	\]
	Because $\Ext^1_R(R/I,R)\neq 0$ and $R$ is local, its annihilator must be a proper ideal; hence it equals $\mathfrak{m}$. By definition,
	\(
	D(I)=\Ann_R\!\big(\Ext^1_R(R/I,R)\big)=\mathfrak{m}.
	\)
\end{proof}

\begin{observation}
	Let $R$ be a ring and suppose $I$ contains a regular element.
	\begin{enumerate}
		\item[(i)] $\tr(I)\subseteq \Ann_R(\Ext^i_R(R/I,R))$ for all $i>1$.
		\item[(ii)] There are examples where $\tr(I)\nsubseteq \Ann_R(\Ext^1_R(R/I,R))$.
	\end{enumerate}
\end{observation}

\begin{proof}
	(i) For $i>1$, from the short exact sequence
	$0\to I\to R\to R/I\to 0$ we obtain
	\(
	\Ext^i_R(R/I,R)\cong \Ext^{i-1}_R(I,R).
	\)
	Since $I$ contains a regular element, we may view $I\subseteq R\subseteq Q(R)$, where $Q(R)$ is the total ring of fractions. By \cite[Theorem 3.3]{dey},
	\(
	\tr(I)\cdot \Ext^j_R(I,R)=0	\)  {for all } 	\(j>0.
	\)
	Taking $j=i-1>0$ yields
	\(
	\tr(I)\cdot \Ext^i_R(R/I,R)=0,
	\)
	so $\tr(I)\subseteq \Ann_R(\Ext^i_R(R/I,R))$.
	
	(ii) By \cite{lindo}, there exists a one-dimensional ring $R$ with a reflexive ideal $I$ such that $\tr(I)$ is not reflexive. Suppose, toward a contradiction, that
	\(
	\tr(I)\subseteq \Ann_R(\Ext^1_R(R/I,R)).
	\)
	As above,
	\(
	\Ext^1_R(R/I,R)\cong I^\ast/R,
	\)
	so
	\[
	\Ann_R(\Ext^1_R(R/I,R))=\Ann_R(I^\ast/R)=(I^\ast)^{-1}.
	\]
	Hence
	\(
	I \subseteq \tr(I) \subseteq (I^\ast)^{-1}=I^{**}.
	\)
	Since $I$ is reflexive, $I^{**}=I$, so $\tr(I)=I$, which is reflexive, a contradiction. 
\end{proof}

By $\ell(I)$ we mean the analytically spread of $I$.

\begin{proposition}\label{5.4}
	Let $(R,\mathfrak m)$ be an analytically unramified local ring and let $I$ be an ideal such that
	\[
	\ell(I_\fp) < \dim(R_\fp)\quad \text{for all } \fp \in \Spec(R)\setminus \min(I).
	\]
	Then
	\(
	\tr(I)\subseteq \overline{I}.
	\)
\end{proposition}

\begin{proof}
	We proceed by induction on $d:=\dim R$.
	
	\medskip
	\noindent\textbf{Base case: $d=1$.}
	In dimension one, it is known that
	\(
	\overline{\tr(I)}=\overline{I}
	\)
	(see \cite[2.4(3)]{lindo}). Hence $\tr(I)\subseteq \overline{I}$.
	
	\medskip
	\noindent\textbf{Inductive step: $d\ge 2$.}
	Assume the statement holds for all analytically unramified local rings of dimension $<d$.
	
	Analytically unramifiedness localizes (see \cite[Prop.~9.1.4]{HS}), so $R_\fq$ is analytically unramified for every prime $\fq$. Since $I\subseteq \tr(I)$, we have
	\(
	\overline{I}\subseteq \overline{\tr(I)}.
	\)
	To prove the reverse inclusion, it suffices (by Fact~\ref{FQ}) to show
	\[
	\overline{IR_\fp}=\overline{\tr(I)R_\fp}
	\quad\text{for every }\fp\in \Ass(\overline{I}).
	\]
	
	\medskip
	\noindent\textbf{(a) $\fp\in \min(I)$:}
Localizing at such a prime, $IR_\fp$ is $\fp R_\fp$-primary.  
If $\Ht(\fp)>1$, then $\dim R_\fp\ge 2$ and $R_\fp$ satisfies $(S_2)$. In this case
	\(
	\grade(IR_\fp,R_\fp)=\grade(\fp R_\fp,R_\fp)\ge 2.
	\)
	For ideals of grade at least two in an $S_2$ ring, the trace is unchanged after localization, so
	\(
	\tr(I)_\fp=\tr(IR_\fp)=IR_\fp,
	\)
	and hence their integral closures agree.
If $\Ht(\fp)=1$, then $\dim R_\fp=1$, and by the one-dimensional case,
	\(
	\overline{IR_\fp}=\overline{\tr(IR_\fp)}.
	\)
	
	\medskip
	\noindent\textbf{(b) $\fp\in \Ass(\overline{I})\setminus \min(I)$:}
If $\fp\ne \fm$, then $\dim R_\fp<d$, and by the inductive hypothesis applied to $R_\fp$ and $IR_\fp$, we get
	\(
	\overline{IR_\fp}=\overline{\tr(IR_\fp)}.
	\)
	Here we use that both trace and integral closure commute with localization.
It remains to rule out $\fp=\fm$. Suppose $\fm\in \Ass(\overline{I})$. Then by a theorem of Burch (see \cite[Prop.~5.4.7]{HS}),
	\(
	\ell(I)=\dim R.
	\)
	But this contradicts the hypothesis, since $\fm\notin \min(I)$ and we assumed
	\(
	\ell(I_\fm)=\ell(I)<\dim(R_\fm)=\dim R.
	\)
	Hence $\fm\notin \Ass(\overline{I})$.
	
	\medskip
	We have shown that $\overline{IR_\fp}=\overline{\tr(IR_\fp)}$ for all $\fp\in \Ass(\overline{I})$. By Fact~\ref{FQ}, this implies
	\(
	\overline{I}=\overline{\tr(I)}.
	\)
	Therefore $\tr(I)\subseteq \overline{I}$.
\end{proof}

	\maketitle
	
\section{Low dimension and small height}

By $\mu(I)$ we mean the minimal number of elements that needs to generate an ideal $I$.

\begin{observation} \label{2e}Suppose $R$ is a domain and that $\mu(I) = 2$.
	
	\begin{enumerate}
		\item[(i)]  
		\(\ann 
		(\Ext^2(R/I,R)) = (I:I^{**}).
		\)
		
		\item[(ii)]  $(I:_R\overline{I})  \subseteq\Ann(\Ext^2(R/I,R))$ if 
	$R$ is analytically unramified and	$\dim (R) = 1$.
	\end{enumerate}
\end{observation} 

\begin{proof} (i): There is an  exact sequence  
	$0\rightarrow I^* \rightarrow R^2 \rightarrow I \rightarrow 0$.
	Recall that $\Tr(-)$ is the Auslander's transpose  of $(-)$. We have
	\[
	\begin{array}{ll}
	\Ext^2(R/I,R) &= \Ext^1(I,R) 
	= \Ext^2(\Tr(I),R) \\
	&= \coker(I\to I^{**})=I^{**}/I,
	\end{array}
	\] 
	and so \(\ann (
	\Ext^2(R/I,R)) = \ann (I^{**}/I)=  (I:I^{**}).
	\)

(ii)	 By \cite[2.14]{corso} we know $\overline{I} \supseteq I^{**}$. This gives
	$$
	r \in (I:\overline{I}) \Rightarrow rI^{**}\subseteq r\overline{I}\subseteq I  \Rightarrow r\in (I:I^{**})
	\stackrel{(i)}=\Ann(\Ext^2(R/I,R)) ,
	$$so $(I:_R\overline{I})  \subseteq\Ann(\Ext^2(R/I,R))$ as claimed.
\end{proof}
By  $\Syz_{i}(M)$ we mean the $i^{th}$  syzygy module of $M$. The following is easy:
\begin{fact}\label{ep}
	Suppose $p:=\pd(M)<	\infty$. Then $\Ext^p(M,R)=\Tr(\Syz_{p-1}(M))$.
\end{fact}

For simplicity, we set $\Omega M:=\Syz_{1}(M)$.

\begin{proposition} 
	If $M$ is perfect module of projective dimension two, then
	\[
	\Ann(\Ext^2(M,R)) = \Ann(M)=\Ann \frac{(\Omega M)^{**}}{\Omega M} = \{r \in R:r(\Omega M)^{**}\subseteq M\}.
	\]
\end{proposition}

\begin{proof}
	The first equality follows from
	\(\Ext^2(\Ext^2(M,R),R) \cong M
	\) and the latter is similar to Observation \ref{2e}(i). In fact, in view of Fact \ref{ep}, we have
	\[
	\begin{array}{ll}
	\Ext^2(\Ext^2(M,R),R)  \cong \Ext^2(\Tr(\Omega M),R) 
 = \coker((\Omega M)\to (\Omega M)^{**}),
	\end{array}
	\]	
	and recall that $\Omega M$ is a submodule of a free module, so torsionless. This means   $\Omega M \to (\Omega M)^{**} $ is an embedding, so the desired  claim is clear.	
\end{proof}
By $(-)^{\vee}$ we mean the  Matlis duality.
\begin{proposition}\label{1d}
	Let $(R,\mathfrak m)$ be a Gorenstein local ring of dimension $d>1$, and let $M$ be a finitely generated $R$-module with $\dim M = 1$. Then
	\(
	\Ann_R\!\big(\Ext^{d-1}_R(M,R)\big)
	= \Ann_R\!(\frac{M}{\Gamma_{\mathfrak m}(M)}).
	\)
	In particular,
	\(
	\Ann_R\!(H^{1}_{\mathfrak m}(M))
	= \Ann_R\!(\frac{M}{\Gamma_{\mathfrak m}(M)}).
	\)
\end{proposition}

\begin{proof}
	Consider the short exact sequence
	\[
	0 \longrightarrow \Gamma_{\mathfrak m}(M)
	\longrightarrow M
	\longrightarrow M/\Gamma_{\mathfrak m}(M)
	\longrightarrow 0 .
	\]
	Applying $\Ext^{\bullet}_R(-,R)$ yields the exact sequence
	\[
	0 = \Ext^{d-2}_R(\Gamma_{\mathfrak m}(M),R)
	\longrightarrow \Ext^{d-1}_R(M/\Gamma_{\mathfrak m}(M),R)
	\longrightarrow \Ext^{d-1}_R(M,R)
	\longrightarrow \Ext^{d-1}_R(\Gamma_{\mathfrak m}(M),R) = 0 .
	\]
	Hence
	\(
	\Ext^{d-1}_R(M,R)
	\cong \Ext^{d-1}_R(M/\Gamma_{\mathfrak m}(M),R).
	\) We may assume $\Gamma_{\mathfrak m}(M)\neq 0$. Then $M/\Gamma_{\mathfrak m}(M)$ has positive depth. Since $\dim M=1$, we get
	\[
	0 < \depth\!\left(M/\Gamma_{\mathfrak m}(M)\right)
	\le \dim\!\left(M/\Gamma_{\mathfrak m}(M)\right)
	= 1,
	\]
	so $\depth(M/\Gamma_{\mathfrak m}(M))=1$.
	By the Auslander–Bridger formula,
	\[
	\sup\{\, i \mid \Ext^i_R(M/\Gamma_{\mathfrak m}(M),R)\neq 0 \,\}
	= \Gdim(M/\Gamma_{\mathfrak m}(M))
	= d - \depth(M/\Gamma_{\mathfrak m}(M))
	= d-1.
	\]
	Moreover,
	\begin{align*}
	\inf\{\, i \mid \Ext^i_R(M/\Gamma_{\mathfrak m}(M),R)\neq 0 \,\}
	&= \grade\!\big(\Ann_R(M/\Gamma_{\mathfrak m}(M)), R\big) \\
	&= \operatorname{ht}\!\big(\Ann_R(M/\Gamma_{\mathfrak m}(M))\big) \\
	&= d - \dim(M/\Gamma_{\mathfrak m}(M)) = d-1 .
	\end{align*}
	Thus $M/\Gamma_{\mathfrak m}(M)$ is quasi-perfect. By \cite[4.11]{supp},
	\[
	\Ann_R\!\big(\Ext^{d-1}_R(M/\Gamma_{\mathfrak m}(M),R)\big)
	= \Ann_R\!\left(M/\Gamma_{\mathfrak m}(M)\right).
	\]
	Using the isomorphism above gives the first equality.
	For the second statement, Matlis duality gives
	$\Ann_R(X) = \Ann_R(X^\vee)$
 for any Artinian module $X.$
	By local duality,
	\(
	H^{1}_{\mathfrak m}(M)^\vee \cong \Ext^{d-1}_R(M,R),
	\)
	and the result follows.
\end{proof}

\begin{observation}
	Let $M$ be a zero-dimensional $R$-module of finite projective dimension $p$. Then
	\[
	\Ann_R\!\big(\Ext^{p}_R(M,R)\big) = \Ann_R(M).
	\]
\end{observation}

\begin{proof}
	The existence of such a module implies that $R$ is Cohen–Macaulay. Let $d=\dim R$. Since $M$ is Artinian,
	\[
	\inf\{\, i \mid \Ext^i_R(M,R)\neq 0 \,\}
	= \grade(\Ann_R(M),R)
	= \depth R
	= d.
	\]
	Also,
	\(
	p = \sup\{\, i \mid \Ext^i_R(M,R)\neq 0 \,\},
	\)
	and by the Auslander–Buchsbaum formula, $p\le d$. Hence $p=d$, so $M$ is perfect. Therefore,
	\(
	\Ann_R\!\big(\Ext^{p}_R(M,R)\big) = \Ann_R(M).
	\)
\end{proof}

Let $\Spec^1(R)$ be the class of all prime ideal $\fp$ with $\Ht(\mathfrak p)=1$.

\begin{corollary}
	Let $P \in \Spec^1(R)$ with $\pd_R(P) < \infty$, and assume $\Char(R)=p>0$. Then for every $q=p^e$,
	\[
	P^{[q]} = P^q = P^{(q)} = (P^q)^\star = \overline{P^q}.
	\]
	In particular,
	\(
	D(P^q) = P^{[q]}.
	\)
\end{corollary}

\begin{proof}
	By \cite[Theorem 1.1(b)]{HH}, we have
	\(
	P^{(q)} \subseteq (P^q)^\star.
	\)
	Thus,
	\[
	P^{[q]} \subseteq P^q \subseteq P^{(q)} \subseteq (P^q)^\star \subseteq \overline{P^q}. \tag{$\ast$}
	\]
	Note that $P^{[q]}$ is $P$-primary. Localizing at $P$, all inclusions in $(\ast)$ become equalities. Hence, by Fact~\ref{FQ}, all containments in $(\ast)$ are equalities.
\footnote{Alternatively, it was conjectured in \cite{ufd} that $P$ is principal, in which case the result is immediate. In fact, one can show $P=\overline{(x)}$ for some $x$, so the conjecture holds over normal rings.}
	From the proof of Example~\ref{ex1}, we have
	\(
	\Ext^1_R\!(\frac{R}{P^q}, R) = (P^q)^\ast / R.
	\)
	Hence
	\(
	\Ann_R\!\left(\Ext^1_R\!(\frac{R}{P^q}, R)\right)
	= (P^q)^{\ast\ast}
	= P^{(q)},
	\)
	which yields the desired conclusion.
\end{proof}

\begin{proposition}\label{6.8}
	Suppose $P \in \Spec(R)$ has analytic spread one. Then 
	\(
	D(P^n)=\overline{P^n}=(P^n)^{**}
	\) for all $n\ge 1$.
\end{proposition}

\begin{proof}
	Since $\ell(P)\ge \Ht(P)$ and $\ell(P)=1$, it follows that $\Ht(P)=1$.
	By a theorem of McAdam \cite{Mcadam},
	\[
	\overline{A}^\ast(P):=\bigcup_{n\ge1}\Ass\!\big(R/\overline{P^n}\big)=\{P\}.
	\]
	Hence $\overline{P^n}$ is $P$-primary for every $n$. Using Fact~\ref{FQ} (double dual agrees with the ideal at all associated primes), we obtain
	\(
	\overline{P^n}=(P^n)^{**}.
	\)
Next we compute $D(P^n)$. Since $\Ht(P)=1$, the ideal $P^n$ contains a non-zerodivisor, so $\grade(P^n,R)=1$. From the short exact sequence
	\[
	0 \longrightarrow P^n \longrightarrow R \longrightarrow R/P^n \longrightarrow 0
	\]
	we get
	\[
	\Ext^1_R(R/P^n,R)\cong \Hom_R(P^n,R)/R=(P^n)^{-1}/R,
	\]
	where $(P^n)^{-1}=\{\,x\in Q(R)\mid xP^n\subseteq R\,\}$.
	Therefore
	\[
	D(P^n)=\Ann_R\!\big(\Ext^1_R(R/P^n,R)\big)
	=\Ann_R\!\big((P^n)^{-1}/R\big)
	=\{\,r\in R \mid r(P^n)^{-1}\subseteq R\,\}.
	\]
	But
	\[
	\{\,r\in R \mid r(P^n)^{-1}\subseteq R\,\}
	=\{\,r\in R \mid r(P^n)^{-1}\subseteq P^n\,\}
	=(P^n)^{**}.
	\]
	Thus $D(P^n)=(P^n)^{**}=\overline{P^n}$, as claimed.
\end{proof}

\begin{corollary}\label{AP}Let $(R, \mathfrak{m},k)$ be a $d$-dimensional quasi-normal ring with $|k| = \infty$.
	Let $P \in \Spec^1(R)$ be such that $\mu(P) \leq d-1$ and $\mu(PR_Q) = 1$ for all $ Q \neq \mathfrak{m}$. Then for any $n$ we have
\(
(P^n)^{**} \subseteq\overline{P^n}	\). In particular,  
	\(
D( P^n ) \subseteq\overline {P^n
	}.\) 
\end{corollary}
\begin{proof}
	By \cite[Page 189]{Mcadam}, we know
	\(
\overline{A}^*(P) =   \{P\}
	\).  Then we use Fact \ref{FQ} to deduce that  
	\(
	P^{(n)} = (P^n)^{**} \subseteq\overline{P^n}	\).\end{proof}

\begin{example}The assumptions of {Corollary} \ref{AP} can be checked. 
Indeed, let 
	$R := k[X,Y,Z,W]/(XY-ZW)$ and $ P = (x,z)$. Then $\mu({P}) = 2$, but $\mu(PR_Q) = 1$  for all   $Q \neq \fm$.
\end{example}

It is well known that a ring $R$ is normal if and only if every principal ideal is integrally closed. The following extends this fact to height–one unmixed ideals.

\begin{corollary}\label{610}
	Let $R$ satisfy Serre’s condition $(R_1)$ and let $I$ be an unmixed ideal of height one. Then $I$ is integrally closed.
\end{corollary}

\begin{proof}
	Clearly $I \subseteq \overline{I}$. We show equality locally at each $\fp \in \Ass(R/I)$. Since $I$ is unmixed of height one, every such $\fp$ has height one. Because $R$ satisfies $(R_1)$, each localization $R_\fp$ is a one–dimensional regular local ring, hence a discrete valuation ring and therefore normal. In a DVR every ideal is principal, hence integrally closed. Thus
	\(
	IR_\fp = \overline{I}\,R_\fp
\text{ for all } \fp \in \Ass(R/I).
	\)
	Applying Fact~\ref{FQ} (an ideal equals another if they agree after localization at all associated primes), we conclude that $I=\overline{I}$.
\end{proof}

\medskip

Now suppose $R$ is one–dimensional and formally unmixed, and let $I\subseteq J \subseteq R$. Recall that
\[
e(I)=e(J) \iff J \subseteq \overline{I},
\qquad
e_{HK}(I)=e_{HK}(J) \iff J \subseteq I^\star.
\]
In 1-dimensional case, we have $e_{HK}(-)=e(-)$, so these criteria imply $\overline{I}=I^\star$. We next give a direct argument showing that both coincide with the Frobenius closure.

\begin{proposition}\label{612}
	Let $R$ be a one–dimensional complete domain with algebraically closed residue field of prime characteristic $p>0$. Then for every ideal $I$,
	\(
	\overline{I}=I^\star=I^F.
	\)
\end{proposition}

\begin{proof}
	Let $q=p^e$ for $e\gg0$. Since $R$ is complete with algebraically closed residue field, the integral closure $\overline{R}$ is module–finite over $R$ and contained in $R^{1/q}$ for $q$ sufficiently large.
	By \cite[Proposition 1.6.1]{HS},
	\(
	\overline{I}=\overline{I\overline{R}} \cap R.
	\)
	Because $\overline{R}$ is a one–dimensional normal domain, it is a Dedekind domain; being local, it is a discrete valuation ring. Hence every ideal of $\overline{R}$ is principal and integrally closed. In particular, $I\overline{R}$ is integrally closed, so
	\[
	\overline{I}=\overline{I\overline{R}} \cap R
	= I\overline{R} \cap R
	\subseteq IR^{1/q} \cap R.
	\]
	
	Take $x \in \overline{I}$. Then $x \in IR^{1/q}$, so we can write
	\[
	x = a_1 x_1 + \cdots + a_t x_t,
	\quad a_i \in R^{1/q},\ x_i \in I.
	\]
	Raising to the $q$th power gives
	\[
	x^q = a_1^q x_1^q + \cdots + a_t^q x_t^q \in I^{[q]},
	\]
	since $a_i^q \in R$. Thus $x \in I^F$, and so $\overline{I} \subseteq I^F$.
The containments
	\(
	I^F \subseteq I^\star \subseteq \overline{I}
	\)
	are standard in one–dimensional rings. Therefore
	\(
	\overline{I}=I^F=I^\star,
	\)
	as required.
\end{proof}

	\maketitle
	
\section{Connections to triviality of vector bundles}

In this section, all rings are assumed to be homomorphic images of a Gorenstein ring (for instance, any complete local ring satisfies this). Moreover, we assume one of the following standing hypotheses without further mention.

\begin{hypothesis}
	\begin{enumerate}
		\item[(i)] $R$ contains an infinite field $\bar{k}$ with $\operatorname{char}(\bar{k})=0$, or
		\item[(ii)] $R$ contains an infinite field $\bar{k}$ with $\operatorname{char}(\bar{k})=p>0$ such that $\bar{k}$ is separable over its prime field.
	\end{enumerate}
\end{hypothesis}

We begin with the following.

\begin{proposition}\label{dv}
	Let $R$ be a Cohen–Macaulay UFD of dimension $d:=\dim R>4$, and let $M$ be a nonfree reflexive $R$–module such that $M_\fp$ is free for all $\fp\in\Spec R$ with $\operatorname{ht}(\fp)\le 3$. If, in addition, $R$ satisfies Serre’s condition $(R_2)$, then
	\(
	\Ann\!\big(\Ext^{d-3}_R(M,R)\big)
	\)
	is a prime ideal of height two.
\end{proposition}

\begin{proof}
	By \cite[Theorem 2.17]{syz}, there exists a prime ideal $P\subset R$ of height two and an exact sequence
	\(
	0 \to F \to  M \to  P \to  0,
	\)
	where $F$ is free. Applying $\Hom_R(-,R)$ yields
	\[
	\Ext^{d-3}_R(M,R)\cong \Ext^{d-3}_R(P,R)\cong \Ext^{d-2}_R(R/P,R).
	\]
	Since $R$ is Cohen–Macaulay and a UFD, it is Gorenstein. Hence
	\(
	\Ext^{d-2}_R(R/P,R)\cong K_{R/P} 
	\). Because $P$ has height two and $R$ satisfies $(R_2)$, $R/P$ is Cohen–Macaulay of dimension $d-2$, and its canonical module has annihilator equal to $P$. Therefore
	\(
	\Ann\!\big(\Ext^{d-3}_R(M,R)\big)=\Ann(K_{R/P})=P,
	\)
	which is a height–two prime ideal.
\end{proof}

\begin{corollary}
	Let $R$ be a Cohen–Macaulay UFD that is strongly normal, with $\dim R=d>4$. If $M$ is a vector bundle with $\depth(M)>3$, then $M$ is free.
\end{corollary}

\begin{proof}
	By \cite[Proposition 1.4.1(b)]{BH}, $M$ is reflexive. Suppose $M$ is not free. Then Proposition~\ref{dv} gives
$E:=\Ext^{d-3}_R(M,R)\cong K_{R/P}\neq 0$
	for some height–two prime $P$. By local duality,
	\(
	H^3_{\fm}(M)\cong E^{\vee}\neq 0,
	\)
	contradicting $\depth(M)>3$. Hence $M$ is free.
\end{proof}

\begin{discussion}[cf.\ \cite{syz}]
	If $V$ is a vector bundle on $\mathbb{P}^n$ of rank $k<n$ that is not a direct sum of line bundles, a theorem of Evans–Griffith shows that at least one of
	\(
\{	H^1(\mathbb{P}^n,V),\dots,H^{k-1}(\mathbb{P}^n,V)\}
	\)
	and similarly one of
	\(
\{	H^{n-1}(\mathbb{P}^n,V),\dots,H^{n-k+1}(\mathbb{P}^n,V)\}
	\)
	is nonzero.
\end{discussion}

\begin{corollary}
	Let $R$ be a Cohen–Macaulay UFD that is strongly normal with $\dim R=d>4$, and let $V=\widetilde{M}$ be a vector bundle on
	\(
	\Spec(R)^\circ:=\Spec(R)\setminus\{\fm\}.
	\)
	If
	\(
	H^2(\Spec(R)^\circ,V)=0,
	\)
	then $V$ is trivial.
\end{corollary}

\begin{proof}
	There is an exact sequence
	\(
	0 \to  C \to  M \to  M^{**} \to  D \to  0,
	\)
	where $C=\Ext^1_R(\Tr(M),R)$ and $D=\Ext^2_R(\Tr(M),R)$. Since $M$ is locally free on $\Spec(R)^\circ$, both $C$ and $D$ have finite length, hence their associated sheaves vanish. Thus
	\(
	V=\widetilde{M}\cong \widetilde{M^{**}},
	\)
	and we may assume $M$ is reflexive.
	By local cohomology,
	\(
	H^2(\Spec(R)^\circ,V)\cong H^3_{\fm}(M).
	\)
	Thus $H^3_{\fm}(M)=0$. If $M$ were not free, Proposition~\ref{dv} would give $\Ext^{d-3}_R(M,R)\neq 0$, and local duality would imply
	\(
	H^3_{\fm}(M)\cong \Ext^{d-3}_R(M,R)^{\vee}\neq 0,
	\)
	a contradiction. Hence $M$ is free and $V$ is trivial.
\end{proof}

\begin{corollary}\label{76}
	Let $R$ be a regular local ring with $\dim R=d>4$, and let $V=\widetilde{M}$ be a vector bundle on $\Spec(R)^\circ$. If
	\(
	H^2(\Spec(R)^\circ,V)=0,
	\)
	then $V$ is trivial.
\end{corollary}

	\maketitle
	
\section{Rigidity of annihilators}\label{s2.1}

It was shown in \cite[\S 1]{corso} that the rigidity of Koszul homology of $\omega_R$  may fail even over $R:=k[x,y,z]/ \fm^n$ (with $n>1$), by showing that $\Ann(H_i(K(x,y; \omega_R)))\nsubseteq\Ann(H_{i+1}(K(x,y; \omega_R)))$ and asked the validity of $$(\ast):\quad\Ann(H_i(K(I;  R)))\subseteq\Ann(H_{i+1}(K(I; R)))? $$ Here, we present both positive and negative sides of this request.

\begin{observation}\label{o212}
	Let $(R,\fm)$ be a zero-dimensional Gorenstein ring. If the  Koszul rigidity $(\ast)$ is valid, then this forces all Koszul homologies to be faithful over $R/I$.
\end{observation}

\begin{proof}
	Let $I$ be an ideal generated by $\underline{x}:=x_1,\ldots,x_n$  and suppose  
	$\Ann(H_i(I; R))\subseteq\Ann(H_{i+1}(I; R))$. Then $$I\subseteq\Ann(H_1(I; R))\subseteq\Ann(H_n(I; R))=\Ann(\ker(R\stackrel{\underline{x}}\lo R^n))=\Ann(0:I)=I,$$where the last equality follows by the Gorenstein assumption (see \cite[Ex. 3.2.15]{BH}). Thus, $\Ann(H_i(I; R))=I$ for all $1\leq i\leq n$.
\end{proof}

Recall that by $\Gdim(-)$ we mean the $G$-dimension.
Over Gorenstein ring any module is of finite $G$-dimension (see \cite{ab}). In particular, the following extends the previous observation.

\begin{corollary}\label{c213}
	Let $I$  be an ideal of finite $\Gdim$ over an artinian local ring $R$.
	If the  Koszul rigidity $(\ast)$ is valid, then   $\Ann_R(H_i(I; R))=I$ for all $1\leq i\leq n$.
\end{corollary}

\begin{proof}
	Dualizing $0\to I\to R\to R/I\to 0$ gives that
	$0\to (0:I)\to R\to I^\ast\to \Ext_{R}^1(R/I,R)=0$, by recalling that $0=\Gdim(R/I)=\sup\{i:\Ext_{R}^1(R/I,R)\neq0\}$. Thus, $I^\ast=R/(0:I)$. Take another duality, gives us $I=I^{\ast\ast}=(R/(0:I))^\ast=(0:(0:I)).$ Now, we continue as {Observation}~\ref{o212}, and get to the desired claim.
\end{proof}

The finiteness of $\Gdim(I)$ is really needed:

\begin{example}
	Assume one of the following: 
	\begin{itemize}
		\item[(a)]   Let  $R:=K[x,y,z]/(x^2,xy,xz,y^2, z^2)$ and $I:=(y,z)$; or
		
		\item[(b)]   Let $R:=K[x,y,z]/(x^3,y^3,z^3,yz^2)$ and   $I:=(x,y)$.
		
	\end{itemize}
	Then 
	$$\Ann_R(H_1(I; R))\subseteq \Ann_R(H_2(I; R))\quad \emph{ but } \Ann_R(H_2(I; R))\neq I.$$In particular,  $\Gdim(I)=\infty.$
\end{example}

\begin{proof}
	(a): One has	$$\Ann_R(H_2(I; R))=\Ann_R(0:I)=\Ann_R(xz,yz)=(x,y,z)=\fm$$
	In particular, $\Ann_R(H_2(I; R))$ contains any proper ideal, as well
	$\Ann_R(H_1(I; R))$. Since, $x\notin (y,z)$, it follows by the previous corollary that $\Gdim(I)=\infty.$
	
	(b): It is easy to see $\Ann_R(H_1(I; R))=I$ and also  $$\Ann_R(H_2(I; R))=\Ann_R(0:I)=\Ann_R((x^2y^2,x^2z^2))=(x,y,z^2).$$
	Since, $z^2\notin (x,y)$, it follows by the previous corollary that $\Gdim(I)=\infty.$
\end{proof}

Now, we go the negative side of question $(\ast)$:
\begin{example}\label{e215}
	Let $k$ be a field and look at 
	Artinian Gorenstein ring
	$$
	R:=k[X,Y,Z]/(X^2-Y^2,Y^2-Z^2,XY,YZ,ZX).
	$$ Put
	$I:=(Y,Z).$ 
	Then
	$\Ann_R H_1(K(I;R))\nsubseteq
	\Ann_R H_2(K(I;R)). $

\end{example}

\begin{proof}  We are going to show
	$ X\in \Ann_R H_1(K(I;R))$
	{but} 
	$X\notin\Ann_R H_2(K(I;R)). $ 
	Put
	$T:=X^2=Y^2=Z^2\in R.$
	The defining relations give
	$XY=XZ=YZ=0$
	and
	$XT=YT=ZT=T^2=0.$ 
	Consequently,
	$R=k\oplus kX\oplus kY\oplus kZ\oplus kT,$
	and the socle of $R$ is $kT$. In particular, $R$ is Artinian
	Gorenstein.
	Consider the Koszul complex on $Y,Z$:
$
	0\to R
	\stackrel{d_2}\longrightarrow R^2
	\stackrel{d_1}\longrightarrow R
	\to0,
	$ 
	where
	$d_2(c)=(-cZ,cY)$  {and} 
	$d_1(a,b)=aY+bZ.$
	We first determine \(Z_1:=\ker(d_1)\).
	Let \((a,b)\in R^2\), and write uniquely
	\begin{itemize}
		\item $
		a=a_0+a_1X+a_2Y+a_3Z+a_4T$ and
		
		\item $
		b=b_0+b_1X+b_2Y+b_3Z+b_4T,$ 
		where all coefficients belong to \(k\).
	\end{itemize}	
	
	Since
	\(XY=XZ=YT=ZT=0\)
	and
	\(Y^2=Z^2=T,\)
	we have
	\(aY=a_0Y+a_2T\)
	and
	\(bZ=b_0Z+b_3T.\)
	Consequently,
	\[
	d_1(a,b)
	=aY+bZ
	=a_0Y+b_0Z+(a_2+b_3)T.
	\]
	
	The elements $\{Y,Z,T\}$ are \(k\)-linearly independent. Hence
	\(d_1(a,b)=0\)
	if and only if
	\[
	a_0=0,\qquad b_0=0,\qquad a_2+b_3=0.
	\]
	Thus every element of \(Z_1\) has the form
	\[
	\begin{aligned}
	(a,b)
	={}a_1(X,0)+b_1(0,X)
	+a_2(Y,-Z) 
	+a_3(Z,0)+b_2(0,Y)
	+a_4(T,0)+b_4(0,T).
	\end{aligned}
	\]
	Therefore, as a \(k\)-vector space,
	\[
	Z_1\subseteq
	k(X,0)+k(0,X)+k(Y,-Z)+k(Z,0)+k(0,Y)
	+k(T,0)+k(0,T).
	\]
	
	Since \(k\subseteq R\), 
	\[
	Z_1\subseteq
	R(X,0)+R(0,X)+R(Y,-Z)+R(Z,0)+R(0,Y)+R(T,0)+R(0,T).
	\]
	The reverse inclusion follows because each of the seven displayed elements are in \(Z_1\) and \(Z_1\) is an  $R$-module.
	On the other hand,
	$
	\operatorname{im}d_2
	=\{(-cZ,cY):c\in R\}.
	$ In particular,
	\begin{itemize}
		\item $
		(-Z,Y)=d_2(1),
		$
		
		\item $
		(-T,0)=(-Z\cdot Z,Z\cdot Y)=d_2(Z),
		$
		\item$
		(0,-T)=(Y\cdot Z,-Y\cdot Y)=d_2(-Y).
		$
	\end{itemize}
	Thus, in $H_1(K(I;R))$, the classes of
	$\{ 
	(Y,-Z),  (T,0),  (0,T)
	\}$
	vanish. Hence $H_1(K(I;R))$ is generated by the classes
	$
	\{	[(X,0)],  [(Z,0)],  [(0,X)],  [(0,Y)]\}.
	$ We now show that $X$ annihilates these generators. Indeed,
	
	\begin{itemize}
		\item $
		X(X,0)=(T,0)=(Z\cdot Z,-Z\cdot Y)=d_2(-Z),
		$
		
		\item $
		X(Z,0)=(0,0),$
		
		\item
		$
		X(0,X)=(0,T)=(-Y \cdot Z,Y \cdot Y)=d_2(Y),
		$	\item$
		X(0,Y)=(0,0).
		$
	\end{itemize}	
	Therefore
	$
	XH_1(K(I;R))=0,
	$	and consequently
	$
	X\in\Ann_R H_1(K(I;R)).
	$
	We next compute the second Koszul homology. Since
	$
	H_2(K(I;R))=0:_R I
	=0:_R(Y,Z),$
	we compute the annihilator. To this end, if
	$$
	a:=a_0+a_1X+a_2Y+a_3Z+a_4T,
	$$
	then
	$aY=a_0Y+a_2T$, and
	$aZ=a_0Z+a_3T.$ Thus
	\begin{center}  $aY=aZ=0$ if and only if
		$
		a_0=a_2=a_3=0.
		$\end{center}
	Hence
	$
	H_2(K(I;R))=(X,T).
	$
	Since $X\cdot X=T\neq0$, we have
	$
	X\notin\Ann_R(X,T).
	$
	Therefore
	$
	X\notin\Ann_R H_2(K(I;R)).
	$
 Gorensteiness guarantees that 
 $
	\Ann_R H_2(K(I;R))
	=\Ann_R(\Ann_R(I))
	=I$ 	(see Corollary \ref{c213}).
\end{proof}

	\maketitle

\end{document}